\documentclass[a4paper,11pt]{article}
\usepackage[top=3.2cm, bottom=3.2cm, left=3.5cm, right=3.5cm]{geometry}
\usepackage[T1]{fontenc}
\usepackage[utf8]{inputenc}
\usepackage{amsmath,amsthm,amssymb,amsfonts,mathtools}
\usepackage{tikz}
\usepackage{xcolor}

\newtheorem{theorem}{Theorem}[section]
\newtheorem{lemma}[theorem]{Lemma}
\newtheorem{proposition}[theorem]{Proposition}
\newtheorem{corollary}[theorem]{Corollary}

\theoremstyle{definition}
\newtheorem{definition}[theorem]{Definition}
\newtheorem{remark}[theorem]{Remark}

\DeclareMathOperator{\comp}{\mathsf{comp}}
\DeclareMathOperator{\rcomp}{\mathsf{rcomp}}
\DeclareMathOperator{\rpath}{\mathsf{rpath}}
\DeclareMathOperator{\leaves}{\mathsf{leaves}}
\DeclareMathOperator{\internal}{\mathsf{internal}}
\DeclareMathOperator{\peak}{\mathsf{peak}}
\DeclareMathOperator{\nonpeak}{\mathsf{nonpeak}}
\DeclareMathOperator{\tr}{\mathsf{tr}}
\DeclareMathOperator{\nontr}{\mathsf{nontr}}
\DeclareMathOperator{\rmin}{\mathsf{rmin}}
\DeclareMathOperator{\nonrmin}{\mathsf{nonrmin}}
\DeclareMathOperator{\rmax}{\mathsf{rmax}}
\DeclareMathOperator{\redge}{\mathsf{redge}}
\DeclareMathOperator{\ledge}{\mathsf{ledge}}
\DeclareMathOperator{\cons}{\mathsf{cons}}
\DeclareMathOperator{\noncons}{\mathsf{noncons}}
\DeclareMathOperator{\blocks}{\mathsf{blocks}}
\DeclareMathOperator{\nonmin}{\mathsf{nonmin}}
\DeclareMathOperator{\subtrees}{\mathsf{sub}}
\DeclareMathOperator{\Av}{Av}
\DeclareMathOperator{\rev}{rev}
\DeclareMathOperator{\corev}{corev}
\DeclareMathOperator{\des}{\mathsf{des}}
\DeclareMathOperator{\asc}{\mathsf{asc}}
\DeclareMathOperator{\ddes}{\mathsf{ddes}}
\DeclareMathOperator{\dasc}{\mathsf{dasc}}
\DeclareMathOperator{\valley}{\mathsf{valley}}
\newcommand{\rc}{\operatorname{\mathsf{rc}}}
\newcommand{\bx}[1]{\boxed{\makebox[1.1em]{$#1$}}}
\newcommand{\ub}[1]{\vphantom{\bx{0}}\makebox[1.1em]{$#1$}}

\newcommand\oeis{{\sc oeis}}

\title{Involution $h$ on Catalan structures}

\author{%
Anders Claesson\thanks{Department of Mathematics,
  University of Iceland, Reykjavik, Iceland;
  \texttt{akc@hi.is}}
\and
Sergey Kitaev\thanks{Department of Mathematics and
  Statistics, University of Strathclyde,
  26 Richmond Street, Glasgow G1 1XH, United Kingdom;
  \texttt{sergey.kitaev@strath.ac.uk}}
\end{tabular}\\[0.5em]\begin{tabular}[t]{c}
Einar Steingr\'{\i}msson\thanks{Department of
  Mathematics and Statistics, University of Strathclyde,
  26 Richmond Street, Glasgow G1 1XH, United Kingdom;
  \texttt{einar.steingrimsson@strath.ac.uk}}
\and
Lintong Wang\thanks{Department of Mathematics,
  Zhejiang Normal University, Jinhua 321004, PR China;
  \texttt{WangLT@zjnu.edu.cn}}\quad\mbox{}
}

\begin{document}
\maketitle
\thispagestyle{empty}

\begin{abstract}
  We define an involution $h$ on Catalan structures through an abstract
  framework, prove an equidistribution theorem for four canonical statistics and
  present a generating function carrying these.  This framework encompasses all
combinatorial structures with a decomposition mirroring the first-return
decomposition of Dyck paths. The fixed points of~$h$ are counted by Catalan
numbers.  Canonical bijections transport the equidistribution to eight well
known concrete families, identifying the canonical statistics with native ones
on each.  In addition to its primary structure, each Catalan structure has a
derived \emph{secondary structure}, and $h$~interchanges primary and secondary
structure. The involution factors as $h = \rev \circ \corev \circ \rev$, where
$\rev$ and $\corev$ are two simpler involutions, and the composition
$M = h \circ \rev$ coincides with Donaghey's automorphism on plane trees. This
yields $M^{-1} = \rev \circ M \circ \rev$ and a period theorem: Iterating
the secondary structure construction produces a sequence that repeats with
period equal to the order of~$M$.  It is an open problem to describe $h$ and
the canonical statistics explicitly on most of the more than two hundred known
families of Catalan structures.
\end{abstract}

\section{Introduction}\label{sec:intro}

The Catalan numbers
$\frac{1}{n+1}\binom{2n}{n}$ form one of the
most ubiquitous sequences in combinatorics, counting,
so far, over two hundred families of objects~\cite{Stanley15}.
Many of these families carry intrinsic algebraic
structure (a notion of indecomposability and a way to
compose objects), and the standard bijections between them often
respect this structure.

In a previous paper~\cite{CKS09} the first three authors defined an involution
$h$ on $\beta(1,0)$-trees; the proof that $h$ is an involution appeared
in~\cite{CKS13}.  Fang~\cite{Fang18} subsequently showed that, under natural
bijections, $h$ corresponds to map duality on non-separable planar maps and to
interval duality on synchronized intervals in the Tamari lattice.  Here we place
the involution $h$ in an abstract setting in which a Catalan structure is
specified by a graded set~$C$, a distinguished empty element~$e$, an
indecomposability operator~$\lambda$, and an associative
composition~$\oplus$. The defining axiom is that every element decomposes
uniquely as $e$, $\lambda(s)$, or $u \oplus v$
(Definition~\ref{def:catalan-structure}). Starting from a primary structure
$(C, e, \lambda, \oplus)$ we derive a secondary structure
$(C,e,\gamma,\ominus)$, define the involution $h$ that interchanges the primary
and secondary structures, characterize its fixed points, and establish an
equidistribution theorem on four statistics defined on the abstract structure
$C$.

We then instantiate the framework on eight concrete Catalan families: rooted
plane trees, Dyck paths, triangulations of convex $n$-gons, permutations
respectively avoiding the patterns $231$ and $321$, binary trees, $2 \times n$
standard Young tableaux, and non-crossing partitions. For each family we
transport the equidistribution result, identifying the transported statistics
with natively defined ones. The resulting equidistribution theorem
(Theorem~\ref{thm:equidistribution}) shows that, on each family, four natively
defined statistics form two equidistributed pairs under~$h$. In the concrete
case of Dyck paths this theorem says that the joint distribution of the number
of returns to the $x$-axis and the number of peaks is the same as that for the
length of the final run of down-steps and one plus the number of double
up-steps.  On some families~$h$ admits a particularly simple description.
On triangulations of the $(n{+}2)$-gon, for instance, $h$ is vertex complement,
relabeling each vertex~$i$ as $n{+}3{-}i$.

Many, if not most, known combinatorial structures counted by the Catalan numbers
have a decomposition mirroring the well known first-return decomposition of Dyck
paths, which automatically makes them an instance of the abstract framework we
define here, so that the equidistribution results mentioned above hold for them.
What is not trivial, however, is how the involution $h$, and the
corresponding equidistributed statistics, manifest in each instance.  Given the
vast number of such structures it seems reasonable to expect many of them to
have intelligible descriptions of $h$ and the equidistributed statistics, which
would be interesting to explore further.  Moreover, it seems likely that $h$
will in many instances translate other statistics between Catalan structures in
intelligible ways.

The abstract framework also sheds light on Donaghey's
automorphism~\cite{Donaghey80} on plane trees, which turns out to be intimately
connected to~$h$.  Donaghey constructs a size-preserving bijection~$M$ on
plane trees (equivalently, on each Catalan family) as $L^{-1} \circ R$, where
$R$ and $L$ are two maps between general and binary bracketings.  Donaghey
himself notes ``the chaotic nature of~$M$''~\cite{Donaghey80}.
The equivalent map $T \circ R$ on ordered forests, where $R$ and $T$ are Knuth's
conjugate and transpose~\cite[Ex.~17, \S7.2.1.6]{Knuth06}, has equally
intractable cycle structure~\cite{Callan07}.  While the cycle structure remains
open (see Section~\ref{sec:open}), our framework gives the factorization
$M = h \circ \rev$, decomposing Donaghey's map into two simpler
involutions: $h$~interchanges primary and secondary structure, and
$\rev$ reverses the first-return decomposition.  This yields
$M^{-1} = \rev \circ M \circ \rev$.  Using a result of Pushkarev and
Byzov~\cite{PushkarevByzov14}, we show that $h$ and $\rev$ generate the free
product $\mathbb{Z}_2 * \mathbb{Z}_2$, so $M$ has unbounded order.  A period
theorem complements this: the iterated secondary structures
$(\lambda_1, \oplus_1), (\lambda_2, \oplus_2), \ldots$ repeat with period equal
to the order of~$M$ (see Section~\ref{sec:donaghey-section}).

The sections of the paper are as follows:
\begin{itemize}
\item Section~\ref{sec:framework}: The abstract Catalan framework, where $h$ and all
identities live.  The map~$h$ is defined in Definition~\ref{def:h-def}.
\item Section~\ref{sec:families}: Eight concrete Catalan families, their operations,
  and concrete descriptions of $h$ where available.
\item Section~\ref{sec:statistics}: Statistics, recursive equations, the
  equidistribution proof, generating functions, and per-family
  corollaries obtained by transporting the canonical statistics
  through canonical bijections, recorded in a dictionary identifying
  each canonical statistic with a native one.
\item Section~\ref{sec:donaghey-section}: The factorization
  $h = \rev \circ \corev \circ \rev$; the map
  $M = h \circ \rev$, which coincides with Donaghey's
  automorphism; the group generated by $h$ and $\rev$, which is the
  free product $\mathbb{Z}_2 * \mathbb{Z}_2$; and iterated secondary
  structures, whose period equals the order of~$M$.
\item Appendix~\ref{sec:bij-verify}: Verification that the seven concrete
  bijections of Section~\ref{sec:bijections} are canonical.
\item Appendix~\ref{sec:dict-verify}: Verification that each entry of the
  dictionary in Section~\ref{sec:dictionary} satisfies the canonical recursion.
\end{itemize}

\section{The abstract Catalan framework}\label{sec:framework}

\begin{definition}\label{def:catalan-structure}
A \emph{Catalan structure} consists of $(C,e,\lambda,\oplus)$,
where $C$ is the disjoint union $\sqcup_{n\geq 0}C_n$, with $C_0=\{e\}$ and
\begin{align*}
  \lambda & \colon C_n\to C_{n+1}\\
  \oplus & \colon C_m\times C_n\to C_{m+n}
\end{align*}
where $\oplus$ is associative with two-sided identity~$e$, and the following
unique decomposition property holds: every $t\in C$ can be written in exactly
one of three forms:
\begin{align*}
  t &= e && && \text{(empty)}\\
  t &= \lambda(s)
    &&\text{for some $s\in C$}
    && \text{(indecomposable)} \\
  t &= u\oplus v
    &&\text{for some indecomposable $u$ and $v\in C\setminus C_0$}
    && \text{(decomposable)}
\end{align*}
\end{definition}

We denote the size of an object $X$, usually understood from the context, by
$|X|$.  In particular, if $t\in C_n$ then $|t|=n$.

By Definition~\ref{def:catalan-structure}, every
nonempty element factors uniquely as
$\lambda(s_1) \oplus \cdots \oplus \lambda(s_k)$ with $k \ge 1$. Peeling from
either end gives a pair of two-case decompositions:
\begin{itemize}
\item First-return decomposition: every nonempty $t$ is
  uniquely $\lambda(u) \oplus v$.
\item Last-return decomposition: every nonempty $t$ is
  uniquely $u \oplus \lambda(v)$.
\end{itemize}

The \emph{free} Catalan structure is realized concretely by the canonical
expressions $e$, $\lambda(t)$, and $\lambda(u) \oplus v$ with $v \neq e$. By
Definition~\ref{def:catalan-structure} and the first-return decomposition, every
element corresponds to exactly one such expression, so the elements may be
identified with these expressions, built up from~$e$ by $\lambda$ and~$\oplus$.
In particular,
\begin{align*}
  C_0 &= \bigl\{\,e\,\bigr\}, \\
  C_1 &= \bigl\{\,\lambda(e)\,\bigr\}, \\
  C_2 &= \bigl\{\,\lambda(\lambda(e)),\; \lambda(e)\oplus\lambda(e)\,\bigr\}, \\
  C_3 &= \bigl\{\,\lambda(\lambda(\lambda(e))),\;
           \lambda(\lambda(e)\oplus\lambda(e)),\; \\
       &\qquad  \lambda(\lambda(e))\oplus\lambda(e),\;
         \lambda(e)\oplus\lambda(\lambda(e)),\;
           \lambda(e)\oplus\lambda(e)\oplus\lambda(e)\,\bigr\}.
\end{align*}
Note that removing the symbols $\lambda$, $e$, and $\oplus$ from the $C_i$ above
we obtain the sets of balanced parentheses well known to be counted by the
Catalan numbers.  The construction described here is the special case
$\lambda = \lambda_1$ of a more general construction on $\beta(1,0)$-trees
in~\cite{CKS09}.

A closely related axiomatization was developed by Brak~\cite{Brak19}, who
observes that each Catalan family carries a \emph{free magma structure}: a
non-associative binary product with unique factorization and an additive
norm. The unique isomorphism between free magmas then gives a ``universal
bijection'' between any two families.  The present framework differs in using an
associative composition $\oplus$ together with a separate unary
operator~$\lambda$, which gives the three-form decomposition directly and makes
the secondary structure~$(\gamma, \ominus)$ and the involution~$h$ natural to
define. Brak does not consider the secondary structure or the involution~$h$.

Each concrete family (defined in Section~\ref{sec:families}) is a Catalan structure,
and the universal property of the free structure ensures that any identity
proved there transfers to every instance. Before introducing the secondary
structure, we give two examples.

Let us view Dyck paths as words over the alphabet $\{u,d\}$.
Assuming $w$, $w_1$ and $w_2$ are Dyck paths, we define
\begin{align*}
  \lambda(w) &= uwd &&\text{(lift the path $w$)} \\
  w_1 \oplus w_2 &= w_1 w_2 &&\text{(concatenate $w_1$ and $w_2$)}.
\end{align*}
Secondly, we consider rooted plane trees, represented as lists of
subtrees hanging from the root. The single-node tree (a leaf)
is the empty list $[\,]$. Assuming $T$, $T_1$ and $T_2$ are
rooted plane trees, we define
\begin{align*}
  \lambda(T) &= [T]
    \quad\text{(new root with } T
    \text{ as sole child)} \\
  T_1 \oplus T_2 &= [c_1, \ldots, c_j,
    c_1', \ldots, c_k']
\end{align*}
where $T_1 = [c_1, \ldots, c_j]$ and
$T_2 = [c_1', \ldots, c_k']$.  In other words, $\oplus$ concatenates the child lists
of the two roots.

\begin{definition}\label{def:secondary}\label{sec:secondary}
Given a Catalan structure $(C,e,\lambda,\oplus)$ we derive
a \emph{secondary structure} $(C,e,\gamma,\ominus)$ on the
same set $C$. Define
\[
  \gamma(t) = t \oplus \lambda(e)
\]
and define $\ominus$ by recursion on the three-form
decomposition of the first argument:
\begin{align*}
       e \ominus v &= v \\
    \lambda(u) \ominus v &= \lambda(u \ominus v) \\
 (t \oplus u) \ominus v &= t \oplus (u \ominus v).
\end{align*}
A decomposable first argument may factor as $t \oplus u$ with $t, u \neq e$ in
more than one way, but $t \oplus (u \ominus v)$ is the same for every such
factorization, so $\ominus$ is well-defined and the third clause may be applied
to any of them. Alternatively, $t$ may be assumed to be indecomposable.
The recursion terminates by induction on the size of the first argument. Indeed, the
second and third clauses reduce to the recursive call $u \ominus v$, whose first
argument $u$ is strictly smaller than $\lambda(u)$ and than $t \oplus u$ (here
$t \neq e$), respectively. The first clause is a base case.
These definitions specialize the constructions of
\cite{CKS09} to the single-$\lambda$ setting.
\end{definition}

On plane trees, $\gamma$ appends a leaf as the
rightmost child of the root, and $t \ominus v$
replaces the leaf at the end of the rightmost-child
chain of~$t$ by~$v$ (Figure~\ref{fig:ominus-trees}).
On Dyck paths, $\gamma(w) = wud$ and
$w_1 \ominus w_2$ replaces the final $ud$ of~$w_1$
with $u w_2 d$.

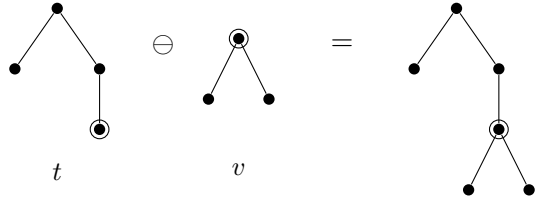
\begin{figure}[ht]
\centering
\begin{tikzpicture}[
  nd/.style={circle, fill, inner sep=1.5pt},
  scale=0.8
]
% === t ===
\begin{scope}[shift={(-4.5, 0)}]
  \node[nd] (tr) at (0, 0) {};
  \node[nd] (tl) at (-0.7, -1) {};
  \node[nd] (tm) at (0.7, -1) {};
  \node[nd] (tb) at (0.7, -2) {};
  \draw (tr) -- (tl);
  \draw (tr) -- (tm);
  \draw (tm) -- (tb);
  \draw (tb) circle (4.4pt);
  \node[font=\small] at (0, -2.7) {$t$};
\end{scope}
% === Skew sum ===
\node at (-2.75, -0.6) {$\ominus$};
% === v ===
\begin{scope}[shift={(-1.5, -0.5)}]
  \node[nd] (vr) at (0, 0) {};
  \node[nd] (vl) at (-0.5, -1) {};
  \node[nd] (vrt) at (0.5, -1) {};
  \draw (vr) -- (vl);
  \draw (vr) -- (vrt);
  \draw (vr) circle (4.4pt);
  \node[font=\small] at (0, -2.16) {$v$};
\end{scope}
% === Arrow ===
\node at (0.21, -0.6) {$=$};
% === t ominus v ===
\begin{scope}[shift={(2.1, 0)}]
  \node[nd] (ur) at (0, 0) {};
  \node[nd] (ul) at (-0.7, -1) {};
  \node[nd] (um) at (0.7, -1) {};
  \node[nd] (ub) at (0.7, -2) {};
  \node[nd] (ubl) at (0.2, -3) {};
  \node[nd] (ubr) at (1.2, -3) {};
  \draw (ur) -- (ul);
  \draw (ur) -- (um);
  \draw (um) -- (ub);
  \draw (ub) -- (ubl);
  \draw (ub) -- (ubr);
  \draw (ub) circle (4.4pt);
\end{scope}
\end{tikzpicture}
\caption{The secondary composition $\ominus$ on plane
trees. The rightmost-child chain of~$t$ ends at a leaf
(circled) and in $t \ominus v$ that leaf is replaced by
the root of~$v$ (also circled), and the rest of~$v$
hangs off below.}
\label{fig:ominus-trees}
\end{figure}

\begin{lemma}\label{lem:ominus-snoc}
For all $t, u, v \in C$,
\[
  (t \oplus \lambda(u)) \ominus v
    = t \oplus \lambda(u \ominus v).
\]
In particular, $\gamma(t) \ominus v = t \oplus \lambda(v)$.
\end{lemma}

\begin{proof}
Applying the third and then the second defining
clause of~$\ominus$ we get
\[
(t \oplus \lambda(u)) \ominus v
 = t \oplus (\lambda(u) \ominus v)
 = t \oplus \lambda(u \ominus v).
\]
Setting $u = e$ and using
$\gamma(t) = t \oplus \lambda(e)$ gives the particular case.
\end{proof}

\begin{proposition}\label{prop:secondary-catalan}
The derived quadruple $(C,e,\gamma,\ominus)$ is itself a
Catalan structure.
\end{proposition}

\begin{proof}
We verify the axioms of
Definition~\ref{def:catalan-structure}.
The grading is preserved: $\gamma\colon C_n \to C_{n+1}$
since $\gamma(t) = t \oplus \lambda(e)$ adds one to the
size, and $\ominus\colon C_m \times C_n \to C_{m+n}$
since each clause of the recursion preserves the total
size.

\emph{Identity.} $e \ominus v = v$ is the first clause
of the definition. The right identity,
$t \ominus e = t$ follows by induction on the
three-form decomposition of~$t$: if $t = e$ it is
immediate; if $t = \lambda(u)$ then
$\lambda(u) \ominus e = \lambda(u \ominus e)
= \lambda(u)$, and if $t = s \oplus u$ then
$(s \oplus u) \ominus e = s \oplus (u \ominus e)
= s \oplus u$.

\emph{Associativity.}
$(t \ominus u) \ominus v = t \ominus (u \ominus v)$
follows by induction on the three-form decomposition
of~$t$. If $t = e$ then both sides equal
$u \ominus v$. If $t = \lambda(s)$ then
\[
(\lambda(s) \ominus u) \ominus v
  = \lambda(s \ominus u) \ominus v
  = \lambda((s \ominus u) \ominus v)
  = \lambda(s \ominus (u \ominus v))
  = \lambda(s) \ominus (u \ominus v),
\]
using the inductive hypothesis on~$s$. If
$t = s \oplus w$ then
\begin{align*}
((s &\oplus w) \ominus u) \ominus v
  = (s \oplus (w \ominus u)) \ominus v
  = s \oplus ((w \ominus u) \ominus v)=\\
  s &\oplus (w \ominus (u \ominus v))
  = (s \oplus w) \ominus (u \ominus v),
\end{align*}
using the inductive hypothesis on~$w$.

\emph{Unique decomposition.} Every nonempty element
has a unique last-return decomposition
$t = u \oplus \lambda(v)$. By
Lemma~\ref{lem:ominus-snoc},
$\gamma(u) \ominus v = u \oplus \lambda(v)$, so
$t = \gamma(u) \ominus v$. If $v = e$ then
$t = \gamma(u)$. If $v \neq e$ then
$t = \gamma(u) \ominus v$ with $\gamma(u)$ and~$v$
both nonempty. Conversely, from
$t = \gamma(u) \ominus v$ we recover
$u \oplus \lambda(v)$, so the secondary decomposition
is unique.
\end{proof}

\begin{lemma}\label{lem:gamma-oplus}
$\gamma(u \oplus v) = u \oplus \gamma(v)$.
\end{lemma}

\begin{proof}
$\gamma(u \oplus v) = (u \oplus v) \oplus \lambda(e)
  = u \oplus (v \oplus \lambda(e))
  = u \oplus \gamma(v)$.
\end{proof}

Conversely, $\lambda$ and $\oplus$ can be recovered from
$\gamma$ and $\ominus$:
\[
  \lambda(t) = \gamma(e) \ominus t
\]
and
\begin{align*}
  t \oplus e &= t \\
  t \oplus \gamma(u) &= \gamma(t \oplus u) \\
  t \oplus (u \ominus v) &= (t \oplus u) \ominus v \qquad (u \neq e).
\end{align*}
Thus, given a primary Catalan structure
$(C,e,\lambda,\oplus)$ we derive a secondary structure
$(C,e,\gamma,\ominus)$ as above (or vice versa).

We now define the maps that relate different Catalan
structures via the unique decomposition of each.
Let $(A,e_A,\lambda_A,\oplus_A)$ and
$(B,e_B,\lambda_B,\oplus_B)$ be two Catalan
structures. A \emph{canonical bijection}
$\varphi : A \to B$ is obtained by
\begin{align*}
  \varphi(e_A) &= e_B \\
  \varphi(\lambda_A(t)) &= \lambda_B(\varphi(t)) \\
  \varphi(u \oplus_A v)
    &= \varphi(u) \oplus_B \varphi(v).
\end{align*}

\begin{definition}\label{def:h}\label{def:h-def}
The \emph{map $h$} is defined by
\begin{align*}
  h(e) &= e \\
  h(\lambda(t)) &= \gamma(h(t)) \\
  h(u \oplus v) &= h(v) \ominus h(u).
\end{align*}
Equivalently, using the first-return decomposition:
\begin{align*}
  h(e) &= e \\
  h(\lambda(u) \oplus v)
    &= h(v) \ominus \gamma(h(u)).
\end{align*}
\end{definition}

\begin{figure}[ht]
\centering
\begin{tikzpicture}[
  nd/.style={circle, fill, inner sep=1.5pt},
  scale=0.78
]
% === COLUMN HEADERS ===
\node[font=\small] at (2.6, 2.2) {$t$};
\node[font=\small] at (9.0, 2.2) {$h(t)$};

% === ROW 1: TREES (y=0.8) ===
% Left: [[e,e], [e,e]]
\begin{scope}[shift={(1.6,1.1)}]
  \node[nd] at (1,0) (tr) {};
  \node[nd] at (0,-0.6) (ta) {};
  \node[nd] at (2,-0.6) (tb) {};
  \node[nd] at (-0.35,-1.2) {};
  \node[nd] at (0.35,-1.2) {};
  \node[nd] at (1.65,-1.2) {};
  \node[nd] at (2.35,-1.2) {};
  \draw (1,0)--(0,-0.6)--(-0.35,-1.2)
        (0,-0.6)--(0.35,-1.2)
        (1,0)--(2,-0.6)--(1.65,-1.2)
        (2,-0.6)--(2.35,-1.2);
\end{scope}
% Right: [[e], [[e],e]]
\begin{scope}[shift={(7.8,1.1)}]
  \draw (1,0)--(0,-0.6)--(0,-1.2)
        (1,0)--(2,-0.6)--(1.6,-1.2)--(1.6,-1.8)
        (2,-0.6)--(2.4,-1.2);
  \foreach \p in {(1,0),(0,-0.6),(0,-1.2),
    (2,-0.6),(1.6,-1.2),(1.6,-1.8),(2.4,-1.2)}
    \node[nd] at \p {};
\end{scope}
\node[font=\small, anchor=east] at (-1.1, 0.2)
  {Trees};

% === ROW 2: DYCK (y=-2.3) ===
% Left: uududduududd
% heights: 0 1 2 1 2 1 0 1 2 1 2 1 0
\begin{scope}[shift={(-0.1,-2.3)}]
  \draw[thick]
    (0,0)--(0.4,0.4)--(0.8,0.8)--(1.2,0.4)
    --(1.6,0.8)--(2.0,0.4)--(2.4,0)
    --(2.8,0.4)--(3.2,0.8)--(3.6,0.4)
    --(4.0,0.8)--(4.4,0.4)--(4.8,0);
  \draw[thin,gray] (0,0)--(4.8,0);
  \node[nd] at (0,0)     (tr) {};
  \node[nd] at (0.4,0.4) (tr) {};
  \node[nd] at (0.8,0.8) (tr) {};
  \node[nd] at (1.2,0.4) (tr) {};
  \node[nd] at (1.6,0.8) (tr) {};
  \node[nd] at (2.0,0.4) (tr) {};
  \node[nd] at (2.4,0)   (tr) {};
  \node[nd] at (2.8,0.4) (tr) {};
  \node[nd] at (3.2,0.8) (tr) {};
  \node[nd] at (3.6,0.4) (tr) {};
  \node[nd] at (4.0,0.8) (tr) {};
  \node[nd] at (4.4,0.4) (tr) {};
  \node[nd] at (4.8,0)   (tr) {};
\end{scope}
% Right: uudduuuddudd
% heights: 0 1 2 1 0 1 2 3 2 1 2 1 0
\begin{scope}[shift={(6.1,-2.3)}]
  \draw[thick]
    (0,0)--(0.4,0.4)--(0.8,0.8)--(1.2,0.4)
    --(1.6,0)--(2.0,0.4)--(2.4,0.8)
    --(2.8,1.2)--(3.2,0.8)--(3.6,0.4)
    --(4.0,0.8)--(4.4,0.4)--(4.8,0);
  \draw[thin,gray] (0,0)--(4.8,0);
  \node[nd] at (0,0)     (tr) {};
  \node[nd] at (0.4,0.4) (tr) {};
  \node[nd] at (0.8,0.8) (tr) {};
  \node[nd] at (1.2,0.4) (tr) {};
  \node[nd] at (1.6,0)   (tr) {};
  \node[nd] at (2.0,0.4) (tr) {};
  \node[nd] at (2.4,0.8) (tr) {};
  \node[nd] at (2.8,1.2) (tr) {};
  \node[nd] at (3.2,0.8) (tr) {};
  \node[nd] at (3.6,0.4) (tr) {};
  \node[nd] at (4.0,0.8) (tr) {};
  \node[nd] at (4.4,0.4) (tr) {};
  \node[nd] at (4.8,0)   (tr) {};
\end{scope}
\node[font=\small, anchor=east] at (-1.1, -2.0)
  {Dyck};

% === ROW 3: TRIANGULATIONS (y=-4.8) ===
% Left: octagon, diags (1,5)(2,4)(2,5)(5,7)(5,8)
\begin{scope}[shift={(2.6,-4.8)}]
  \foreach \i/\a in {1/202.5, 2/247.5, 3/292.5,
    4/337.5, 5/22.5, 6/67.5, 7/112.5, 8/157.5} {
    \coordinate (L\i) at (\a:1.3);
    \node[nd] at (L\i) {};
    \node[font=\tiny] at (\a:1.6) {\i};
  }
  \foreach \i/\j in {1/2,2/3,3/4,4/5,5/6,6/7,7/8,8/1}
    { \draw (L\i)--(L\j); }
  \foreach \i/\j in {1/5,2/4,2/5,5/7,5/8}
    { \draw (L\i)--(L\j); }
\end{scope}
% Right: octagon, diags (1,4)(2,4)(4,7)(4,8)(5,7)
\begin{scope}[shift={(9.0,-4.8)}]
  \foreach \i/\a in {1/202.5, 2/247.5, 3/292.5,
    4/337.5, 5/22.5, 6/67.5, 7/112.5, 8/157.5} {
    \coordinate (R\i) at (\a:1.3);
    \node[nd] at (R\i) {};
    \node[font=\tiny] at (\a:1.6) {\i};
  }
  \foreach \i/\j in {1/2,2/3,3/4,4/5,5/6,6/7,7/8,8/1}
    { \draw (R\i)--(R\j); }
  \foreach \i/\j in {1/4,2/4,4/7,4/8,5/7}
    { \draw (R\i)--(R\j); }
\end{scope}
\node[font=\small, anchor=east] at (-1.1, -4.8)
  {Tri};

% === ROW 4: AV(231) (y=-7.2) ===
\node[font=\small] at (2.6, -7.2)
  {$3\;1\;2\;6\;4\;5$};
\node[font=\small] at (9.0, -7.2)
  {$2\;1\;6\;4\;3\;5$};
\node[font=\small, anchor=east] at (-1.1, -7.2)
  {$\Av(231)$};

% === ROW 5: AV(321) (y=-8.2) ===
\node[font=\small] at (2.6, -8.2)
  {$3\;1\;2\;6\;4\;5$};
\node[font=\small] at (9.0, -8.2)
  {$2\;1\;4\;6\;3\;5$};
\node[font=\small, anchor=east] at (-1.1, -8.2)
  {$\Av(321)$};

% === ROW 6: BINARY TREES (y=-10.2) ===
% Left: ((None,(None,None)),((None,(None,None)),None))
%       root
%      /    \
%     A      C
%      \    /
%       B  D
%           \
%            E
\begin{scope}[shift={(1.6,-9.3)}]
  \node[nd] at (1,0) {};       % root
  \node[nd] at (0.2,-0.6) {};  % A
  \node[nd] at (0.6,-1.2) {};  % B
  \node[nd] at (1.8,-0.6) {};  % C
  \node[nd] at (1.4,-1.2) {};  % D
  \node[nd] at (1.8,-1.8) {};  % E
  \draw (1,0)--(0.2,-0.6) (0.2,-0.6)--(0.6,-1.2)
        (1,0)--(1.8,-0.6) (1.8,-0.6)--(1.4,-1.2)
        (1.4,-1.2)--(1.8,-1.8);
\end{scope}
% Right: ((None,None),(((None,None),(None,None)),None))
%       root
%      /    \
%     F      G
%           /
%          H
%         / \
%        I   J
\begin{scope}[shift={(7.8,-9.3)}]
  \node[nd] at (1,0) {};       % root
  \node[nd] at (0.4,-0.6) {};  % F
  \node[nd] at (1.6,-0.6) {};  % G
  \node[nd] at (1.2,-1.2) {};  % H
  \node[nd] at (0.8,-1.8) {};  % I
  \node[nd] at (1.6,-1.8) {};  % J
  \draw (1,0)--(0.4,-0.6)
        (1,0)--(1.6,-0.6) (1.6,-0.6)--(1.2,-1.2)
        (1.2,-1.2)--(0.8,-1.8)
        (1.2,-1.2)--(1.6,-1.8);
\end{scope}
\node[font=\small, anchor=east] at (-1.1, -10.2)
  {Binary};

% === ROW 7: SYT (y=-12.2) ===
% Left: ([1,2,4,7,8,10],[3,5,6,9,11,12])
\begin{scope}[shift={(0.6,-12.5)}]
  \foreach \i in {0,...,5} {
    \draw (\i*0.6, 0) rectangle (\i*0.6+0.6, 0.5);
    \draw (\i*0.6, -0.5) rectangle (\i*0.6+0.6, 0);
  }
  \foreach \i/\v in {0/1,1/2,2/4,3/7,4/8,5/{10}}
    \node[font=\tiny] at (\i*0.6+0.3, 0.25) {\v};
  \foreach \i/\v in {0/3,1/5,2/6,3/9,4/{11},5/{12}}
    \node[font=\tiny] at (\i*0.6+0.3, -0.25) {\v};
\end{scope}
% Right: ([1,2,5,6,7,10],[3,4,8,9,11,12])
\begin{scope}[shift={(7.0,-12.5)}]
  \foreach \i in {0,...,5} {
    \draw (\i*0.6, 0) rectangle (\i*0.6+0.6, 0.5);
    \draw (\i*0.6, -0.5) rectangle (\i*0.6+0.6, 0);
  }
  \foreach \i/\v in {0/1,1/2,2/5,3/6,4/7,5/{10}}
    \node[font=\tiny] at (\i*0.6+0.3, 0.25) {\v};
  \foreach \i/\v in {0/3,1/4,2/8,3/9,4/{11},5/{12}}
    \node[font=\tiny] at (\i*0.6+0.3, -0.25) {\v};
\end{scope}
\node[font=\small, anchor=east] at (-1.1, -12.2)
  {SYT};

% === ROW 8: NCP (y=-14.7) ===
% Left: {1,3},{2},{4,6},{5}
\begin{scope}[shift={(1.7,-14.6)}]
  \foreach \i in {1,...,6} {
    \coordinate (N\i) at (\i*0.75 - 1.925, 0);
    \node[nd] at (N\i) {};
    \node[font=\tiny, below=2pt] at (N\i) {\i};
  }
  \draw (N1) to[bend left=70] (N3);
  \draw (N4) to[bend left=70] (N6);
\end{scope}
% Right: {1,2},{3,4,6},{5}
\begin{scope}[shift={(8.1,-14.6)}]
  \foreach \i in {1,...,6} {
    \coordinate (M\i) at (\i*0.75 - 1.925, 0);
    \node[nd] at (M\i) {};
    \node[font=\tiny, below=2pt] at (M\i) {\i};
  }
  \draw (M1) to[bend left=85] (M2);
  \draw (M3) to[bend left=85] (M4);
  \draw (M4) to[bend left=70] (M6);
\end{scope}
\node[font=\small, anchor=east] at (-1.1, -14.7)
  {NCP};

\end{tikzpicture}
\caption{The involution $h$ on a size-$6$ Catalan
object, shown simultaneously across all eight families.
Corresponding objects are related by canonical
bijections.}
\label{fig:h-example}
\end{figure}
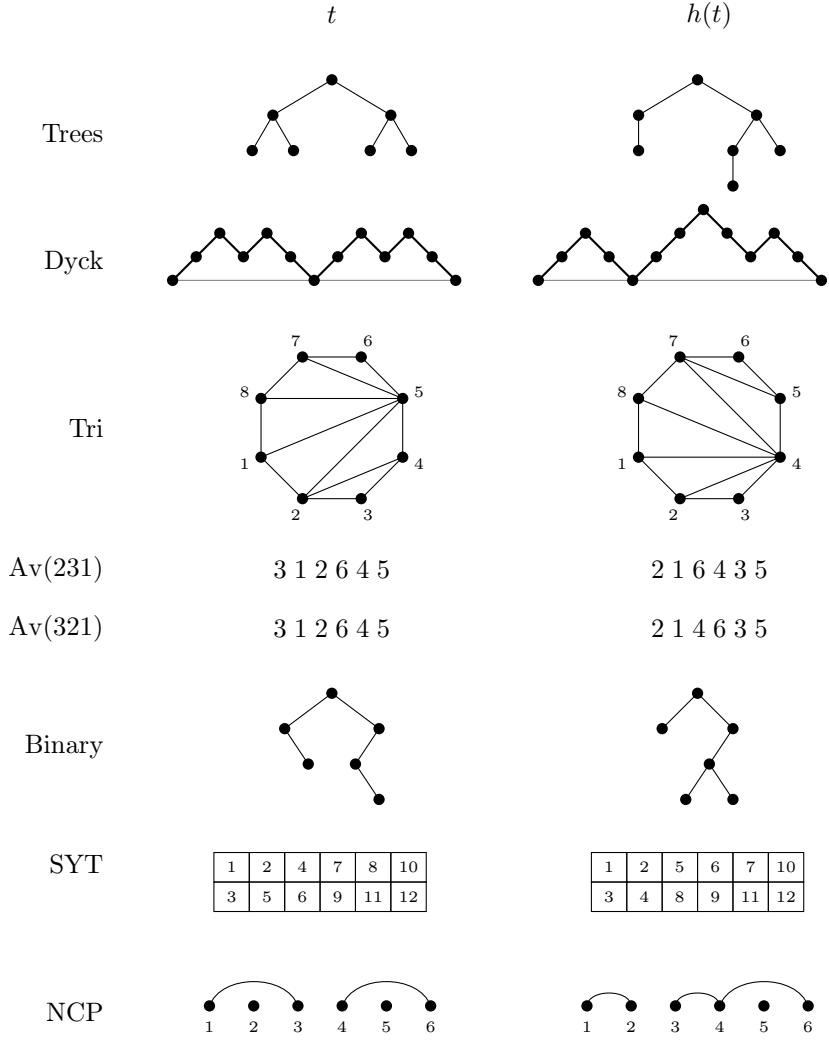

\begin{lemma}[Last-return form]\label{lem:h-last-return}
The map $h$ is also characterized by
\begin{align*}
  h(e) &= e \\
  h(u \oplus \lambda(v))
    &= h(v) \oplus \lambda(h(u))
\end{align*}
\end{lemma}

\begin{proof}
By Lemma~\ref{lem:ominus-snoc},
\[
  h(u \oplus \lambda(v))
    = h(\lambda(v)) \ominus h(u)
    = \gamma(h(v)) \ominus h(u)
    = h(v) \oplus \lambda(h(u)).  \qedhere
\]
\end{proof}

For example, on plane trees,
$h([\,]) = [\,]$,
$h([e]) = \gamma(e) = [e]$,
and $h([e, e, e]) = \lambda(\lambda(\lambda(e)))
= [[[e]]]$: the tree with three children of the root
maps to the path of depth three.
Figure~\ref{fig:h-example} shows $h$ acting
simultaneously on all eight families for a size-$6$
example.

\begin{lemma}\label{lem:h-gamma}
$h(\gamma(t)) = \lambda(h(t))$ for all $t$.
\end{lemma}

\begin{proof}
By Lemma~\ref{lem:ominus-snoc},
$h(\gamma(t)) = h(t \oplus \lambda(e))
  = h(e) \oplus \lambda(h(t)) = \lambda(h(t))$.
\end{proof}

\begin{theorem}\label{thm:h-involution}
$h$ is an involution.
\end{theorem}

\begin{proof}
By induction on size. The base case
$h(h(e)) = h(e) = e$ is immediate. For
$t = u \oplus \lambda(v)$:
\[
  h^2(u \oplus \lambda(v))
    = h(h(v) \oplus \lambda(h(u)))
    = h(h(u)) \oplus \lambda(h(h(v)))
    = u \oplus \lambda(v),
\]
using Lemma~\ref{lem:h-last-return} twice and the
induction hypothesis.\footnote{The last identity in the
  proof of \cite[Theorem~5]{CKS13} should read
  $\cdots = h^2(u) \oplus h^2(v) = u \oplus v$.}
\end{proof}

\begin{corollary}
  $h(u \ominus v) = h(v) \oplus h(u)$ for all $u, v$.
\end{corollary}

\begin{proof}
By the definition of $h$ and Theorem~\ref{thm:h-involution},
\[
  h(h(v) \oplus h(u)) = h(h(u)) \ominus h(h(v)) = u \ominus v.
\]
Applying $h$ to both sides and using Theorem~\ref{thm:h-involution} again,
$h(u \ominus v) = h(v) \oplus h(u)$.
\end{proof}

\begin{proposition}\label{prop:h-fixed-points}
The fixed points of $h$ are $e$ and the elements
\[
  u \oplus \lambda(h(u)) \quad \text{for } u \in C.
\]
In particular, there are $|C_k|$ fixed points of size $2k+1$ (the $k$-th Catalan
number) and none of positive even size.
\end{proposition}

\begin{proof}
The empty element $e$ is fixed. For nonempty $t$,
write $t = u \oplus \lambda(v)$ in its last-return
decomposition. By Lemma~\ref{lem:h-last-return},
$h(t) = h(v) \oplus \lambda(h(u))$, which is also
in last-return form. By uniqueness of the
last-return decomposition, $h(t) = t$ is equivalent
to the pair of equations $h(v) = u$ and $h(u) = v$ and,
since $h$ is an involution
(Theorem~\ref{thm:h-involution}), either one
implies the other, so both are equivalent to the
single condition $h(u) = v$. The fixed points of
size $n$ are thus parametrized by $u \in C_k$ via
$u \mapsto u \oplus \lambda(h(u))$ when $n = 2k+1$,
giving $|C_k|$ fixed points, and there are none
when $n$ is even.
\end{proof}

On the more general $\beta(1,0)$-trees (rooted plane
trees with certain positive integer labels on the
nodes), the fixed-point structure is richer. Kitaev and de
Mier~\cite{KdM14} show that, in addition to the
analogue of the $u \oplus \lambda(h(u))$ construction
above, a second recursive family of fixed points
arises from the interaction of labels with the
rightmost path. They enumerate all fixed points and
obtain the sequence A006013 in the \oeis~\cite{OEIS}.

\begin{proposition}[Transport]\label{prop:transport}
Let $P(t_1, \ldots, t_k) = Q(t_1, \ldots, t_k)$ be an
equation between expressions built from $e$, $\lambda$,
$\oplus$, $\gamma$, and $\ominus$. If the equation
holds for all $t_1, \ldots, t_k$ in the free Catalan
structure, then it holds in every Catalan structure.
\end{proposition}

\begin{proof}
For any Catalan structure $(C,e,\lambda,\oplus)$,
structural recursion on the three-form decomposition
of the free structure defines a unique map~$\varphi$
satisfying $\varphi(e) = e$,
$\varphi(\lambda(t)) = \lambda(\varphi(t))$, and
$\varphi(u \oplus v) = \varphi(u) \oplus \varphi(v)$.
The same recursion applied to the three-form
decomposition of~$C$ defines an inverse
$\varphi^{-1}$ from~$C$ to the free structure, and
$\varphi \circ \varphi^{-1} = \varphi^{-1} \circ \varphi = \operatorname{id}$
by induction on the respective decompositions.
Since~$\gamma$ and $\ominus$ are defined in terms of
$\lambda$ and $\oplus$, $\varphi$ preserves them as
well. Any equation between such expressions that
holds in the free structure therefore holds in~$C$.
\end{proof}

The involution~$h$ is defined by recursion on
$\lambda$ and $\oplus$, so the same canonical bijection
$\varphi$ preserves it. Consequently, every identity
and corollary involving~$h$ proved in the free Catalan
structure, including the involution property, the
fixed-point characterization, and the equidistribution
theorem below, holds uniformly across all families.

\section{Concrete Catalan families}\label{sec:families}

We now instantiate the abstract framework on eight
concrete families. For each, we define $e$, $\lambda$,
and $\oplus$, and describe the resulting $\gamma$ and
$\ominus$ in combinatorial terms. We also give
concrete descriptions of $h$ where available.

\subsection{Rooted plane trees}\label{sec:trees}

A rooted plane tree (henceforth \emph{plane tree}) is
represented as a list of the children of the root. The
single-node tree (a leaf) has an empty child list. The
indecomposable trees are those with exactly one child of the
root.
\begin{itemize}
\item $e = [\,]$, the single-node tree (empty child
  list).
\item $\lambda(T) = [T]$: a new root with $T$ as its
  sole child.
\item $T_1 \oplus T_2$ concatenates the child lists of
  the two roots under a common root.
\item $\gamma(T)$ appends a leaf as the rightmost child
  of the root.
\item $T_1 \ominus T_2$ replaces the leaf at the end
  of the rightmost-child chain of~$T_1$ by~$T_2$.
\end{itemize}

The involution $h$ has the following description on plane trees.
If $T = [c_1, \ldots, c_k]$, then the root of $h(T)$ has
as children those of $h(c_k)$, followed by one new child.
That child's children are those of $h(c_{k-1})$, followed
by another new child, and so on until the deepest new node,
whose children are those of $h(c_1)$ followed by a leaf.

\subsection{Dyck paths}\label{sec:dyck}

Dyck paths of semilength $n$ are words over $\{u,d\}$ of
length $2n$ such that each prefix has at least as many
$u$'s as $d$'s and equally many in total.
The indecomposable Dyck paths are those of the form $uwd$
(touching the $x$-axis only at the endpoints).
\begin{itemize}
\item $e = \epsilon$, the empty path.
\item $\lambda(w) = uwd$: enclose $w$ in an up-step and
  a down-step (lift $w$ one unit).
\item $w_1 \oplus w_2 = w_1 w_2$: concatenation.
\item $\gamma(w) = wud$: append a peak.
\item $w_1 \ominus w_2$: replace the final peak~$ud$ of~$w_1$ with
  $uw_2d$. Every nonempty Dyck path has at least one peak, so this is
  well-defined, the base case being $\epsilon \ominus w_2 = w_2$.
\end{itemize}

Writing $D = D_1 \cdots D_k$ with $D_i = u w_i d$,
the involution acts by
\[
  h(D) = h(w_k)\, u\, h(w_{k-1})\, u
    \cdots u\, h(w_1)\, u\, d^k.
\]
In particular, $h(\epsilon) = \epsilon$ and
$h(uwd) = h(w) ud$.

\subsection{Triangulations}\label{sec:tri}

A triangulation of the convex $(n{+}2)$-gon is a maximal
set of non-crossing diagonals, with vertices labeled
$1, 2, \ldots, n{+}2$ clockwise. It has $n{-}1$
diagonals, $n{+}2$ boundary edges, and $n$ triangles, and we say it has \emph{size}~$n$.
An indecomposable triangulation is one in which vertices
$2$ and $n{+}2$ belong to a common triangle
(equivalently, vertex~$1$ has no incident diagonals).
\begin{itemize}
\item $e$ is the single edge $(1,2)$.
\item $\lambda(T)$: relabel each vertex $i \to i{+}1$
  in~$T$, then add a new vertex~$1$ adjacent to~$2$
  and to~$n{+}3$.
\item $T_1 \oplus T_2$: identify the edge
  $(1, n_1{+}2)$ of~$T_1$ with the edge $(1,2)$
  of~$T_2$, shifting $T_2$'s labels accordingly.
\item $\gamma(T) = T \oplus \lambda(e)$: append a
  triangle sharing the edge $(1, n{+}2)$.
\item $T_1 \ominus T_2$: if $T_1$ is decomposable, let $V$
  be its last indecomposable summand and write
  $T_1 = U \oplus V$; then
  $T_1 \ominus T_2 = U \oplus (V \ominus T_2)$. If $T_1$
  is indecomposable, write $T_1 = \lambda(T_1')$ (remove
  vertex~$1$ and its unique incident triangle, relabel
  $i \mapsto i{-}1$) and set
  $T_1 \ominus T_2 = \lambda(T_1' \ominus T_2)$.
\end{itemize}

Figure~\ref{fig:tri-operations} illustrates these four operations on small
triangulations.

\begin{figure}[htbp]
\centering
% ---- lambda ----
\begin{tikzpicture}[scale=0.8, nd/.style={circle, fill, inner sep=1.3pt},
  glue/.style={double, double distance=0.8pt}]
\begin{scope}[shift={(-2.5,0)}]
  \node[nd] (t1) at (1,0) {}; \node[nd] (t2) at (-0.5,0.866) {};
  \node[nd] (t3) at (-0.5,-0.866) {};
  \draw (t1)--(t2)--(t3)--(t1);
  \node[right=1pt,font=\small] at (t1) {$1$};
  \node[above left=0pt,font=\small] at (t2) {$2$};
  \node[below left=0pt,font=\small] at (t3) {$3$};
  % \node[font=\small] at (0,-1.7) {$T$};
\end{scope}
\draw[->] (-0.85,0)--(0.05,0); \node[above,font=\small] at (-0.4,0.07) {$\lambda$};
\begin{scope}[shift={(1.6,0)}]
  \node[nd] (s1) at (1,0) {}; \node[nd] (s2) at (0,1) {};
  \node[nd] (s3) at (-1,0) {}; \node[nd] (s4) at (0,-1) {};
  \draw (s2)--(s3)--(s4); \draw (s2)--(s4); \draw[dashed] (s2)--(s1)--(s4);
  \node[right=1pt,font=\small] at (s1) {$1$};
  \node[above=1pt,font=\small] at (s2) {$2$};
  \node[left=1pt,font=\small] at (s3) {$3$};
  \node[below=1pt,font=\small] at (s4) {$4$};
  % \node[font=\small] at (0,-1.7) {$\lambda(T)$};
\end{scope}
\end{tikzpicture}\qquad\qquad
% \par\vspace{6mm}
% ---- gamma ----
\begin{tikzpicture}[scale=0.8, nd/.style={circle, fill, inner sep=1.3pt},
  glue/.style={double, double distance=0.8pt}]
\begin{scope}[shift={(-2.5,0)}]
  \node[nd] (t1) at (1,0) {}; \node[nd] (t2) at (-0.5,0.866) {};
  \node[nd] (t3) at (-0.5,-0.866) {};
  \draw (t1)--(t2)--(t3)--(t1);
  \node[right=1pt,font=\small] at (t1) {$1$};
  \node[above left=0pt,font=\small] at (t2) {$2$};
  \node[below left=0pt,font=\small] at (t3) {$3$};
  % \node[font=\small] at (0,-1.7) {$T$};
\end{scope}
\draw[->] (-0.85,0)--(0.05,0); \node[above,font=\small] at (-0.4,0.07) {$\gamma$};
\begin{scope}[shift={(1.6,0)}]
  \node[nd] (s1) at (1,0) {}; \node[nd] (s2) at (0,1) {};
  \node[nd] (s3) at (-1,0) {}; \node[nd] (s4) at (0,-1) {};
  \draw (s1)--(s2)--(s3); \draw (s1)--(s3); \draw[dashed] (s1)--(s4)--(s3);
  \node[right=1pt,font=\small] at (s1) {$1$};
  \node[above=1pt,font=\small] at (s2) {$2$};
  \node[left=1pt,font=\small] at (s3) {$3$};
  \node[below=1pt,font=\small] at (s4) {$4$};
  % \node[font=\small] at (0,-1.7) {$\gamma(T)$};
\end{scope}
\end{tikzpicture}
\par\vspace{3mm}
% ---- oplus ----
\begin{tikzpicture}[scale=0.8, nd/.style={circle, fill, inner sep=1.3pt},
  glue/.style={double, double distance=0.8pt}]
\begin{scope}[shift={(-5.2,0)}]
  \node[nd] (a1) at (1,0) {}; \node[nd] (a2) at (0,1) {};
  \node[nd] (a3) at (-1,0) {}; \node[nd] (a4) at (0,-1) {};
  \draw (a1)--(a2)--(a3)--(a4); \draw (a1)--(a3); \draw[glue] (a4)--(a1);
  \node[right=1pt,font=\small] at (a1) {$1$};
  \node[above=1pt,font=\small] at (a2) {$2$};
  \node[left=1pt,font=\small] at (a3) {$3$};
  \node[below=1pt,font=\small] at (a4) {$4$};
  % \node[font=\small] at (0,-1.7) {$T_1$};
\end{scope}
\node at (-3.15,0) {$\oplus$};
\begin{scope}[shift={(-1.6,0)}]
  \node[nd] (b1) at (1,0) {}; \node[nd] (b2) at (-0.5,0.866) {};
  \node[nd] (b3) at (-0.5,-0.866) {};
  \draw (b1)--(b3)--(b2); \draw[glue] (b1)--(b2);
  \node[right=1pt,font=\small] at (b1) {$1$};
  \node[above left=0pt,font=\small] at (b2) {$2$};
  \node[below left=0pt,font=\small] at (b3) {$3$};
  % \node[font=\small] at (0,-1.7) {$T_2$};
\end{scope}
\node at (0.50,0) {$=$};
\begin{scope}[shift={(2.4,0)}]
  \node[nd] (c1) at (1,0) {}; \node[nd] (c2) at (0.309,0.951) {};
  \node[nd] (c3) at (-0.809,0.588) {}; \node[nd] (c4) at (-0.809,-0.588) {};
  \node[nd] (c5) at (0.309,-0.951) {};
  \draw (c1)--(c2)--(c3)--(c4)--(c5)--(c1); \draw (c1)--(c3); \draw[glue] (c1)--(c4);
  \node[right=1pt,font=\small] at (c1) {$1$};
  \node[above=1pt,font=\small] at (c2) {$2$};
  \node[above left=0pt,font=\small] at (c3) {$3$};
  \node[below left=0pt,font=\small] at (c4) {$4$};
  \node[below=1pt,font=\small] at (c5) {$5$};
  % \node[font=\small] at (0,-1.7) {$T_1\oplus T_2$};
\end{scope}
\end{tikzpicture}
\par\vspace{3mm}
% ---- ominus ----
\begin{tikzpicture}[scale=0.78, nd/.style={circle, fill, inner sep=1.3pt},
  glue/.style={double, double distance=0.8pt}]
% T_1 = square(1,3); diagonal (1,3) doubled (the U | V split)
\begin{scope}[shift={(-7.8,0)}]
  \node[nd] (a1) at (1,0) {}; \node[nd] (a2) at (0,1) {};
  \node[nd] (a3) at (-1,0) {}; \node[nd] (a4) at (0,-1) {};
  \draw (a1)--(a2)--(a3)--(a4)--(a1); \draw[glue] (a1)--(a3);
  \node[right=1pt,font=\small] at (a1) {$1$};
  \node[above=1pt,font=\small] at (a2) {$2$};
  \node[left=1pt,font=\small] at (a3) {$3$};
  \node[below=1pt,font=\small] at (a4) {$4$};
  \node[font=\small] at (0,-2.1) {$T_1$};
\end{scope}
\node at (-6.0,0) {$\ominus$};
% T_2 = triangle
\begin{scope}[shift={(-4.6,0)}]
  \node[nd] (b1) at (1,0) {}; \node[nd] (b2) at (-0.5,0.866) {};
  \node[nd] (b3) at (-0.5,-0.866) {};
  \draw (b1)--(b2)--(b3)--(b1);
  \node[right=1pt,font=\small] at (b1) {$1$};
  \node[above left=0pt,font=\small] at (b2) {$2$};
  \node[below left=0pt,font=\small] at (b3) {$3$};
  \node[font=\small] at (0,-2.1) {$T_2$};
\end{scope}
\node at (-2.6,0) {$=$};
% U = triangle; glue edge (1,3) doubled
\begin{scope}[shift={(-1.0,0)}]
  \node[nd] (u1) at (1,0) {}; \node[nd] (u2) at (-0.5,0.866) {};
  \node[nd] (u3) at (-0.5,-0.866) {};
  \draw (u1)--(u2)--(u3); \draw[glue] (u3)--(u1);
  \node[right=1pt,font=\small] at (u1) {$1$};
  \node[above left=0pt,font=\small] at (u2) {$2$};
  \node[below left=0pt,font=\small] at (u3) {$3$};
  \node[font=\small] at (0,-2.1) {$U$};
\end{scope}
\node at (0.9,0) {$\oplus$};
% V ominus T_2 = square(2,4) = lambda(T_2); glue edge (1,2) doubled
\begin{scope}[shift={(2.7,0)}]
  \node[nd] (w1) at (1,0) {}; \node[nd] (w2) at (0,1) {};
  \node[nd] (w3) at (-1,0) {}; \node[nd] (w4) at (0,-1) {};
  \draw (w2)--(w3)--(w4)--(w1); \draw (w2)--(w4); \draw[glue] (w1)--(w2);
  \node[right=1pt,font=\small] at (w1) {$1$};
  \node[above=1pt,font=\small] at (w2) {$2$};
  \node[left=1pt,font=\small] at (w3) {$3$};
  \node[below=1pt,font=\small] at (w4) {$4$};
  \node[font=\small] at (0,-2.1) {$V\ominus T_2$};
\end{scope}
\node at (4.75,0) {$=$};
% result pentagon; seam (1,3) doubled
\begin{scope}[shift={(6.6,0)}]
  \node[nd] (c1) at (1,0) {}; \node[nd] (c2) at (0.309,0.951) {};
  \node[nd] (c3) at (-0.809,0.588) {}; \node[nd] (c4) at (-0.809,-0.588) {};
  \node[nd] (c5) at (0.309,-0.951) {};
  \draw (c1)--(c2)--(c3)--(c4)--(c5)--(c1); \draw[glue] (c1)--(c3); \draw (c3)--(c5);
  \node[right=1pt,font=\small] at (c1) {$1$};
  \node[above=1pt,font=\small] at (c2) {$2$};
  \node[above left=0pt,font=\small] at (c3) {$3$};
  \node[below left=0pt,font=\small] at (c4) {$4$};
  \node[below=1pt,font=\small] at (c5) {$5$};
\end{scope}
\end{tikzpicture}
\caption{The four triangulation operations on small instances. Dashed edges are
those added by $\lambda$ and $\gamma$; double edges those identified by $\oplus$
and $\ominus$.}
\label{fig:tri-operations}
\end{figure}
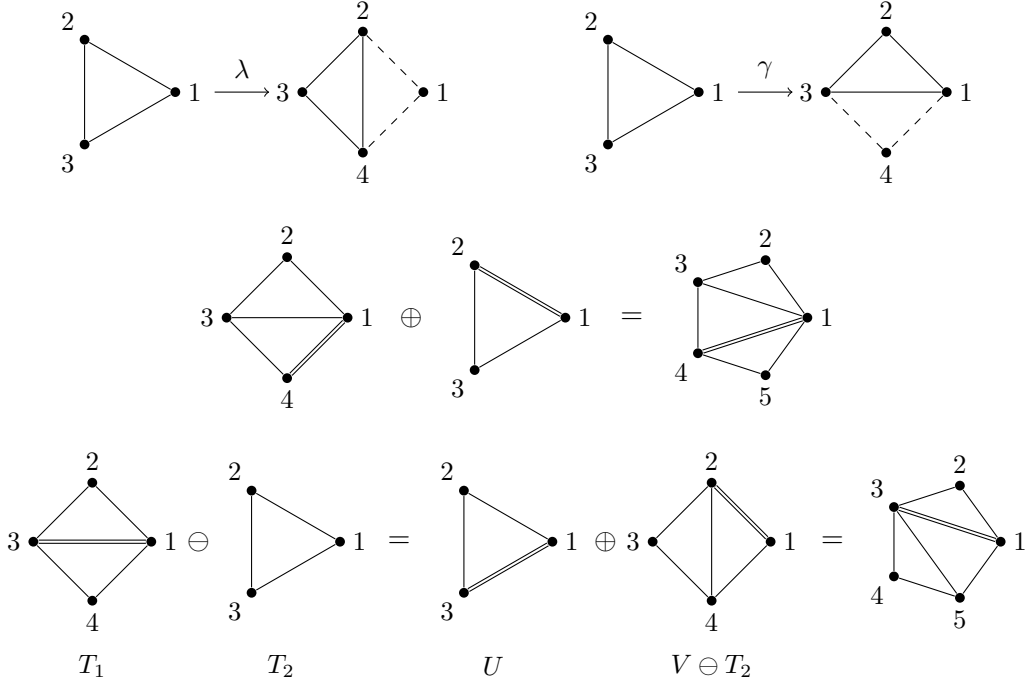

\begin{proposition}[Vertex complement]\label{prop:vertex-complement}
On triangulations of the $(n{+}2)$-gon, $h$ acts by
relabeling each vertex $i$ as $n+3-i$. Equivalently,
edge $(i,j)$ maps to $(n+3-i, n+3-j)$.
\end{proposition}

\begin{figure}[ht]
\centering
\begin{tikzpicture}[
  nd/.style={circle, fill, inner sep=1.5pt},
  scale=0.75
]
% === Left hexagon: T ===
\begin{scope}[shift={(-4.0, 0)}]
  \node[nd] (v1) at (2, 0) {};
  \node[nd] (v2) at (1, 1.732) {};
  \node[nd] (v3) at (-1, 1.732) {};
  \node[nd] (v4) at (-2, 0) {};
  \node[nd] (v5) at (-1, -1.732) {};
  \node[nd] (v6) at (1, -1.732) {};
  \draw (v1) -- (v2) -- (v3) -- (v4) -- (v5) -- (v6) -- (v1);
  \draw (v1) -- (v3);
  \draw (v3) -- (v5);
  \draw (v1) -- (v5);
  \node[right=2pt, font=\small] at (v1) {$1$};
  \node[above right, font=\small] at (v2) {$2$};
  \node[above left, font=\small] at (v3) {$3$};
  \node[left, font=\small] at (v4) {$4$};
  \node[below left, font=\small] at (v5) {$5$};
  \node[below right, font=\small] at (v6) {$6$};
  \node[font=\small] at (0, -2.6) {$T$};
\end{scope}
% === Arrow ===
\draw[->, thick] (-0.5, 0) -- (0.5, 0);
\node[above, font=\small] at (0, 0.05) {$h$};
% === Right hexagon: h(T) = \overline{T} ===
\begin{scope}[shift={(4.0, 0)}]
  \node[nd] (w1) at (2, 0) {};
  \node[nd] (w2) at (1, 1.732) {};
  \node[nd] (w3) at (-1, 1.732) {};
  \node[nd] (w4) at (-2, 0) {};
  \node[nd] (w5) at (-1, -1.732) {};
  \node[nd] (w6) at (1, -1.732) {};
  \draw (w1) -- (w2) -- (w3) -- (w4) -- (w5) -- (w6) -- (w1);
  \draw (w6) -- (w4);
  \draw (w4) -- (w2);
  \draw (w6) -- (w2);
  \node[right=2pt, font=\small] at (w1) {$1$};
  \node[above right, font=\small] at (w2) {$2$};
  \node[above left, font=\small] at (w3) {$3$};
  \node[left=2pt, font=\small] at (w4) {$4$};
  \node[below left, font=\small] at (w5) {$5$};
  \node[below right, font=\small] at (w6) {$6$};
  \node[font=\small] at (0, -2.6) {$\overline{T}$};
\end{scope}
\end{tikzpicture}
\caption{Vertex complement $i \mapsto 7 - i$ on a
hexagon triangulation ($n = 4$). The diagonals
$\{(1,3),(3,5),(1,5)\}$ of $T$, forming the central
triangle $(1,3,5)$, map to $\{(6,4),(4,2),(6,2)\}$, the
central triangle to $(2,4,6)$.}
\label{fig:vertex-complement}
\end{figure}
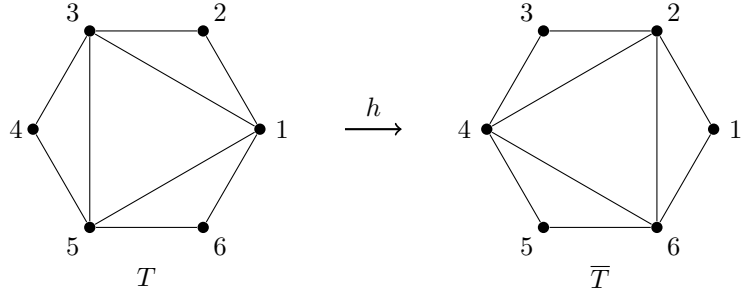

\begin{proof}
Write $\overline{T}$ for the vertex-complement of a
triangulation~$T$, that is,\ the triangulation obtained by
relabeling each vertex~$i$ as $n{+}3{-}i$.
We show $\overline{T} = h(T)$ by induction on size,
using Lemma~\ref{lem:h-last-return}:
$h(e) = e$ and
$h(u \oplus \lambda(v)) = h(v) \oplus \lambda(h(u))$.

\emph{Base case.} The empty triangulation $e$ is the edge
$(1,2)$. The complement sends $1 \mapsto 2, 2 \mapsto 1$,
preserving $e$. So $\overline{e} = e$.

\emph{Inductive step.} Let $T = T_u \oplus \lambda(T_v)$ be a triangulation of
size $n = p{+}q{+}1$ on the $(n{+}2)$-gon, with $|T_u| = p$ and $|T_v| = q$. The
vertex structure in the combined polygon is $T_u$ on
$\{1, \ldots, p{+}2\}$; $T_v$'s vertex $i$ at combined vertex $p{+}i{+}1$ (via
the $\lambda$-shift $i \mapsto i{+}1$ and the $\oplus$-shift $j \mapsto p{+}j$
for $j \geq 2$); and the triangle $(1, p{+}2, n{+}2)$ from $\lambda$.  The
complement $i \mapsto n{+}3{-}i$ acts as follows:

(a) $T_v$'s vertex $i$ (at combined $p{+}i{+}1$) maps
to $(n{+}3) - (p{+}i{+}1) = q{+}3{-}i$, which is the
complement on the $(q{+}2)$-gon. So the $T_v$
sub-triangulation becomes $\overline{T_v}$ on vertices
$\{1, \ldots, q{+}2\}$.

(b) $T_u$'s vertex $i$ (at combined $i$) maps to
$n{+}3{-}i = (p{+}3{-}i) + (q{+}1)$, which is the
complement on the $(p{+}2)$-gon, shifted by $q{+}1$.
So the $T_u$ sub-triangulation becomes $\overline{T_u}$
placed on vertices $\{q{+}2, \ldots, n{+}2\}$.

(c) The triangle $(1, p{+}2, n{+}2)$ maps to
$(n{+}2, q{+}2, 1) = (1, q{+}2, n{+}2)$.

Now compare with
$\overline{T_v} \oplus \lambda(\overline{T_u})$. Here
$\lambda(\overline{T_u})$ is on the $(p{+}3)$-gon with
triangle $(1, 2, p{+}3)$, and $\overline{T_u}$ on vertices
$\{2, \ldots, p{+}3\}$. After the $\oplus$-shift
(vertex $j \mapsto q{+}j$ for $j \geq 2$),
$\overline{T_u}$ sits on $\{q{+}2, \ldots, n{+}2\}$ and
the triangle becomes $(1, q{+}2, n{+}2)$, matching
(b) and (c). By the inductive hypothesis,
$\overline{T_u} = h(T_u)$ and $\overline{T_v} = h(T_v)$, so
$\overline{T} = h(T_v) \oplus \lambda(h(T_u)) = h(T)$.
\end{proof}

\begin{corollary}
Since $h$ is vertex complement, its fixed points are the
symmetric triangulations. There are none for even polygons,
and $|C_{k}|$ for the $(2k+3)$-gon. The axis of symmetry
runs through vertex $k+2$ and the midpoint of the opposite
edge $\{1, 2k+3\}$, so one side determines the other.
\end{corollary}

\subsection{$231$-avoiding permutations}\label{sec:av231}

A permutation $\pi$ of $[n] = \{1, \ldots, n\}$ is $231$-avoiding if no
subsequence of length~$3$ has elements in the relative order of size $2,3,1$. We
write $\Av(231)$ for the set of $231$-avoiding permutations.  A position $i$ is
a \emph{left-to-right maximum} of~$\pi$ if $\pi(j) < \pi(i)$ for all $j < i$,
and a \emph{right-to-left minimum} of~$\pi$ if $\pi(i) < \pi(j)$ for all
$j > i$.  The indecomposable $231$-avoiding permutations are those that begin
with their maximum, that is,\ $\pi(1) = n$.  The \emph{direct sum} of two
permutations $\pi_1$ and $\pi_2$ is the concatenation
$\pi_1\oplus\pi_2=\pi_1\pi_2'$ where $\pi_2'(i)=\pi_2(i)+|\pi_1|$.  Closure
under $\lambda$ and $\oplus$ is immediate; prepending a new maximum cannot
create a 231 pattern, and direct sum preserves avoidance.
\begin{itemize}
\item $e = \epsilon$, the empty permutation.
\item $\lambda(\pi) = (n{+}1) \pi$: prepend a new
  maximum.
\item $\pi_1 \oplus \pi_2$: direct sum.\footnote{The
  notation $\oplus$ is used both for the abstract
  Catalan operation and for the direct sum of
  permutations; on $\Av(231)$ and $\Av(321)$ the two
  coincide.}
\item $\gamma(\pi) = \pi (n{+}1)$: append a new
  fixed point.
\item $\pi_1 \ominus \pi_2$: let $x = \pi_1(n)$ be the
  last entry of~$\pi_1$. Shift all values ${\geq}\, x$
  in~$\pi_1$ up by $|\pi_2|$ to open a gap, then
  append $\pi_2 + (x{-}1)$, where $\pi_2 + (x{-}1)$
  denotes the word obtained from $\pi_2$ by adding
  $x-1$ to each entry.
\end{itemize}
That $|\Av(231) \cap S_n| = |C_n|$ is
classical~\cite[\S2.2.1, Ex.~4--5]{Knuth68}.
No closed form for $h$ on $\Av(231)$ is known; $h$ is
computed via the generic last-return recursion of
Lemma~\ref{lem:h-last-return}.

\subsection{$321$-avoiding permutations}\label{sec:av321}

A permutation $\pi$ of $[n]$ is $321$-avoiding if no
subsequence of length~$3$ is decreasing. We write
$\Av(321)$ for the set of $321$-avoiding permutations.
The indecomposable $321$-avoiding permutations are
those whose only direct-sum component is the whole
permutation. Unlike the other
families, the $\lambda$ operation on $\Av(321)$ is not
immediate from the definition. The construction below
is the map~$\beta$ of~\cite{CK08}; we prove closure
and bijectivity.
\begin{itemize}
\item $e = \epsilon$, the empty permutation.
\item $\lambda(\pi)$: insert $n{+}1$ at the position
  of~$1$, then cyclically shift a distinguished subset
  of positions (described below).
\item $\pi_1 \oplus \pi_2$: direct sum.
\item $\gamma(\pi) = \pi (n{+}1)$: append a new
fixed point.
\item $\pi_1 \ominus \pi_2$: write
  $\pi_1 = \alpha_1 \oplus \cdots \oplus \alpha_k$ in its
  components and replace the last, $\alpha_k$, by
  $\alpha_k \ominus \pi_2$. Since $\alpha_k$ is
  indecomposable, $\alpha_k = \lambda(u)$ and
  $\alpha_k \ominus \pi_2 = \lambda(u \ominus \pi_2)$, with
  $\lambda$ as defined below.
\end{itemize}

More precisely, $\lambda(\pi)$ works as follows.

\paragraph{Pseudocode.}
\begin{enumerate}
\item Let $n = |\pi|$. If $n = 0$, return $1$.
\item Let $i$ be the position of $1$ in $\pi$.
\item Identify the ``boxed'' positions in $\pi$: these
  are the left-to-right maxima of~$\pi$ to the right of
  position~$i$ that are not also right-to-left minima
  of~$\pi$.
\item Insert $n{+}1$ at position $i$, shifting subsequent
  entries right, and add this new position to the boxed
  set.
\item Cyclically shift the values at the boxed positions
  one step to the left: the value at each boxed position
  moves to the previous boxed position, and the first
  boxed value wraps to the last boxed position.
\item Return the result.
\end{enumerate}

An example is shown in
Figure~\ref{fig:321-lambda}: $\pi = 241357698$ has
$i = 3$ and boxed positions $\{6, 8\}$ (the
left-to-right maximum values $7$ and $9$ to the right
of $1$ are not right-to-left minima). Inserting $10$
at position $3$ and boxing it gives boxed set
$\{3, 7, 9\}$; the cyclic left-shift of $10, 7, 9$
produces $7, 9, 10$.

\begin{figure}[ht]
\centering
\[
\arraycolsep=1.5pt
\begin{array}{r*{10}{c}}
\pi \;= & \ub{2} & \ub{4} & \ub{1} & \ub{3} & \ub{5}
  & \bx{7} & \ub{6} & \bx{9} & \ub{8} & \\[4pt]
  & \ub{2} & \ub{4} & \bx{10} & \ub{1} & \ub{3}
  & \ub{5} & \bx{7} & \ub{6} & \bx{9} & \ub{8} \\[4pt]
\lambda(\pi) \;= & \ub{2} & \ub{4} & \bx{7} & \ub{1}
  & \ub{3} & \ub{5} & \bx{9} & \ub{6} & \bx{10}
  & \ub{8}
\end{array}
\]
\caption{The map $\lambda$ on $\Av(321)$. First row:
box the left-to-right maxima to the right of~$1$ that
are not right-to-left minima. Second row: insert
$n{+}1$ at the position of~$1$ and box it. Third row:
cyclically shift the boxed values one step to the
left.}
\label{fig:321-lambda}
\end{figure}

\begin{lemma}\label{lem:321-nonlmax}
A permutation $\tau$ avoids $321$ if and only if the
subsequence of values at positions that are not
left-to-right maxima is increasing.
\end{lemma}

\begin{proof}
($\Rightarrow$) Suppose $b < c$ are positions, neither
a left-to-right maximum, with $\tau(b) > \tau(c)$.
Since $b$ is not a left-to-right maximum, there exists
$a < b$ with $\tau(a) > \tau(b) > \tau(c)$, a $321$
pattern.

($\Leftarrow$) If $\tau$ contains a $321$ pattern
$\tau(a) > \tau(b) > \tau(c)$ with $a < b < c$, then
neither $b$ nor $c$ is a left-to-right maximum,
so the subsequence of values
at non-left-to-right-maximum positions has a descent.
\end{proof}

\begin{proposition}\label{prop:321-closure}
If $\pi$ avoids $321$, then $\lambda(\pi)$ avoids $321$.
\end{proposition}

\begin{proof}
For $n = 0$ the result is trivial
($\lambda(e) = 1$). Assume $n \geq 1$ and let $p$ be
the position of $1$ in $\pi$. Write
$\sigma = \lambda(\pi)$. We show that the values of
$\sigma$ at positions that are not left-to-right
maxima form an increasing subsequence; by
Lemma~\ref{lem:321-nonlmax}, this implies
$\sigma \in \Av(321)$.

\emph{Left-to-right maxima of $\sigma$.} Let
$j_1 < \cdots < j_m$ be the boxed positions in $\pi$
and write $b_0 = p,\; b_k = j_k + 1$ for
$1 \leq k \leq m$. The left-to-right maxima of $\sigma$
are exactly the positions $j < p$ together with
$b_0, \ldots, b_m$. Indeed, every $j < p$ is a
left-to-right maximum of $\pi$ (otherwise some
$k < j$ with $\pi(k) > \pi(j) > 1 = \pi(p)$ gives a
$321$ pattern); these positions retain their values
in $\sigma$ and remain left-to-right maxima. The
shifted values
$\sigma(b_0) = \pi(j_1),\; \sigma(b_1) = \pi(j_2),
\;\ldots,\; \sigma(b_{m-1}) = \pi(j_m),\;
\sigma(b_m) = n{+}1$ are strictly increasing (since
$j_1 < \cdots < j_m$ are successive left-to-right
maxima of $\pi$), so each $b_k$ is a left-to-right
maximum of $\sigma$.
Conversely, a non-boxed position $j > p$ has
$\sigma(j) = \pi(j{-}1)$, and the nearest preceding
boxed position $b_k$ satisfies
$\sigma(b_k) > \sigma(j)$ (since $\sigma(b_k)$ is
either $\pi(j_{k+1})$ with $j_{k+1} > j{-}1$ and
$j_{k+1}$ a left-to-right maximum, or $n{+}1$), so
$j$ is not a left-to-right maximum.

\emph{Values at non-left-to-right-maximum positions
are increasing.} These positions in $\sigma$ are the
non-boxed positions $> p$, with values
$\pi(p), \pi(q_1), \ldots, \pi(q_r)$ where
$p < q_1 < \cdots < q_r$ are the non-boxed positions
$> p$ in $\pi$. Since $\pi(p) = 1$ is minimal, it
suffices to show $\pi(q_i) < \pi(q_j)$ for $i < j$.
If $q_i$ is not a left-to-right maximum of $\pi$,
some $k < q_i$ has $\pi(k) > \pi(q_i)$, so
$\pi(q_i) \geq \pi(q_j)$ would yield a $321$ pattern
at $k < q_i < q_j$. If $q_i$ is a left-to-right
maximum of $\pi$ (hence also a right-to-left minimum,
since non-boxed left-to-right maxima to the right of
$1$ are right-to-left minima), then
$\pi(q_i) < \pi(\ell)$ for all $\ell > q_i$.
\end{proof}

\begin{proposition}\label{prop:321-bijectivity}
The map $\lambda$ is a bijection from $\Av(321)$ of
size $n$ to the set of indecomposable $321$-avoiding
permutations of size $n{+}1$.
\end{proposition}

\begin{proof}
\emph{Image.} Let $p$ be the position of $1$ in $\pi$. By
Proposition~\ref{prop:321-closure}, $\sigma = \lambda(\pi)$
avoids $321$, and $|\sigma| = n{+}1$, so it remains to verify
that $\sigma$ is indecomposable. A component boundary at a
position $k$ with $1 \leq k < n{+}1$ would require
$\{1, \ldots, k\}$ to be exactly the first $k$ values of
$\sigma$. For $k \leq p$, the value $1$ lies at position
$p{+}1 > k$ (it is shifted one place right by the insertion of
$n{+}1$ at position $p$, and position $p{+}1$ is not boxed), so
$1$ is not among the first $k$ values, a contradiction. For
$k > p$, the left-to-right-maxima claim in the closure proof
gives a boxed position $b \leq k$ with $\sigma(b) > k$ (the
largest boxed position not exceeding $k$), so the first $k$
values contain $\sigma(b) \notin \{1, \ldots, k\}$, again a
contradiction. Hence $\sigma$ is indecomposable.

\emph{Injectivity.} Given $\sigma = \lambda(\pi)$, we
recover $\pi$ uniquely. As noted, the value $1$ appears in
$\sigma$ at position $p{+}1$, so $p$ is determined by $\sigma$. The boxed
positions used in the cyclic shift are exactly the
left-to-right maxima of $\sigma$ at positions $\geq p$
(by the left-to-right-maxima claim in the closure
proof), identifiable
from $\sigma$ alone. Reversing the cyclic shift at
these positions and deleting $n{+}1$ at position~$p$
recovers $\pi$ uniquely.

\emph{Cardinality.}
It is well known that $|\Av(321) \cap S_n| = |C_n|$.
The
generating function for indecomposable elements of
any Catalan-counted family is $xC(x)$ (since
$C(x) = 1/(1-xC(x))$), so the number of
indecomposable $321$-avoiders of size $n{+}1$ is also
$|C_n|$. An injective map between two sets of equal
finite cardinality is a bijection.
\end{proof}

No closed form for $h$ on $\Av(321)$ is known; $h$ is
computed via the generic last-return recursion of
Lemma~\ref{lem:h-last-return}.

\subsection{Binary trees}\label{sec:binary}

A binary tree is either the empty tree $\varnothing$
(no nodes) or a node with a left and a right subtree, that is, a tuple $(L,R)$,
each itself a binary tree. An indecomposable binary tree is one of the
form $(T, \varnothing)$, that is,\ a root whose
right subtree is empty. The \emph{right spine} is the
maximal sequence of right edges from the root.
\begin{itemize}
\item $e = \varnothing$, the empty tree.
\item $\lambda(T) = (T, \varnothing)$: make $T$ the left
  child of a new root with no right child.
\item $T_1 \oplus T_2$ attaches $T_2$ at the bottom of
  the right spine of~$T_1$. Explicitly,
  $\varnothing \oplus B = B$ and
  $(T, B_1) \oplus B_2 = (T, B_1 \oplus B_2)$.
\item $\gamma(B) = B \oplus (\varnothing, \varnothing)$:
  adjoin a leaf at the end of the right spine.
\item $B \ominus C$: descend from the root, at each node
  taking the right subtree if it is nonempty and the
  left subtree otherwise, until reaching a leaf; replace
  that leaf by $\lambda(C) = (C, \varnothing)$. If
  $B = \varnothing$, the result is~$C$.
\end{itemize}

No closed form for $h$ on binary trees is known; $h$ is
computed via the generic last-return recursion of
Lemma~\ref{lem:h-last-return}.

\subsection{$2 \times n$ standard Young tableaux}\label{sec:syt}

A $2 \times n$ standard Young tableaux (SYT) is a filling of a $2 \times n$ grid
with $1, 2, \ldots, 2n$ such that rows and columns
increase. An indecomposable SYT is one that cannot be split
as $T_1 \oplus T_2$ with both parts nonempty, that is,\ there
is no $k$ with $0 < k < n$ such that the first $k$ columns
contain exactly $\{1, \ldots, 2k\}$.
\begin{itemize}
\item $e = ([\,], [\,])$, the empty tableau.
\item $\lambda(T)$: increment all entries by~$1$,
  prepend~$1$ to the top row, and append $2n{+}2$ to
  the bottom row.
\item $T_1 \oplus T_2$: shift all entries of~$T_2$ up
  by $2|T_1|$ and adjoin the columns to the right.
\item $\gamma(T)$: append the column
  $(2n{+}1, 2n{+}2)$.
\item $T_1 \ominus T_2$: if $T_1$ is empty, $T_1 \ominus T_2 = T_2$.
  If $T_1$ is decomposable, let $V$
  be its last indecomposable summand and $U$ the rest, so
  $T_1 = U \oplus V$; then
  $T_1 \ominus T_2 = U \oplus (V \ominus T_2)$. If $T_1$
  is indecomposable, write $T_1 = \lambda(T_1')$ (remove
  $1$ from the top row and $2n$ from the bottom row,
  subtract $1$ from all remaining entries) and set
  $T_1 \ominus T_2 = \lambda(T_1' \ominus T_2)$.
\end{itemize}
See Figure~\ref{fig:syt-ops}.

\begin{figure}[ht]
\centering
\begin{tikzpicture}[baseline=-2.75pt, scale=0.5]
% === T ===
\begin{scope}
  \foreach \i in {0,1} {
    \draw (\i,0) rectangle (\i+1,1);
    \draw (\i,-1) rectangle (\i+1,0);
  }
  \foreach \i/\v in {0/1,1/2}
    \node[font=\scriptsize] at (\i+0.5,0.5) {\v};
  \foreach \i/\v in {0/3,1/4}
    \node[font=\scriptsize] at (\i+0.5,-0.5) {\v};
  \node[font=\small] at (1, -1.8) {$T$};
\end{scope}
% === lambda(T) ===
\begin{scope}[shift={(5, 0)}]
  \foreach \i in {0,1,2} {
    \draw (\i,0) rectangle (\i+1,1);
    \draw (\i,-1) rectangle (\i+1,0);
  }
  \foreach \i/\v in {0/1,1/2,2/3}
    \node[font=\scriptsize] at (\i+0.5,0.5) {\v};
  \foreach \i/\v in {0/4,1/5,2/6}
    \node[font=\scriptsize] at (\i+0.5,-0.5) {\v};
  \node[font=\small] at (1.5, -1.8) {$\lambda(T)$};
\end{scope}
% === gamma(T) ===
\begin{scope}[shift={(11, 0)}]
  \foreach \i in {0,1,2} {
    \draw (\i,0) rectangle (\i+1,1);
    \draw (\i,-1) rectangle (\i+1,0);
  }
  \foreach \i/\v in {0/1,1/2,2/5}
    \node[font=\scriptsize] at (\i+0.5,0.5) {\v};
  \foreach \i/\v in {0/3,1/4,2/6}
    \node[font=\scriptsize] at (\i+0.5,-0.5) {\v};
  \node[font=\small] at (1.5, -1.8) {$\gamma(T)$};
\end{scope}
\end{tikzpicture}
\caption{Actions of $\lambda$ and $\gamma$ on a
$2 \times 2$ SYT $T$. The operation $\lambda$ increments
all entries by $1$, prepends $1$ to the top row, and
appends $2n{+}2 = 6$ to the bottom row, whereas $\gamma$ appends
the column $(2n{+}1, 2n{+}2) = (5, 6)$.}
\label{fig:syt-ops}
\end{figure}
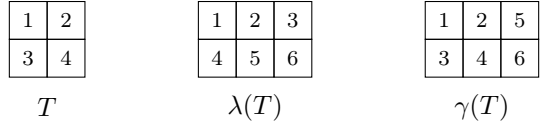

No closed form for $h$ on $2 \times n$ SYT is known;
$h$ is computed via the generic last-return recursion
of Lemma~\ref{lem:h-last-return}. For example,
\[
  h\!\left[\,
  \begin{tikzpicture}[baseline=-2.75pt, scale=0.45]
    \foreach \i in {0,1,2} {
      \draw (\i,0) rectangle (\i+1,1);
      \draw (\i,-1) rectangle (\i+1,0);
    }
    \foreach \i/\v in {0/1,1/2,2/4}
      \node[font=\scriptsize] at (\i+0.5,0.5) {\v};
    \foreach \i/\v in {0/3,1/5,2/6}
      \node[font=\scriptsize] at (\i+0.5,-0.5) {\v};
  \end{tikzpicture}
  \,\right]
  = (\gamma\circ h)\!\left[\,
  \begin{tikzpicture}[baseline=-2.75pt, scale=0.45]
    \foreach \i in {0,1} {
      \draw (\i,0) rectangle (\i+1,1);
      \draw (\i,-1) rectangle (\i+1,0);
    }
    \foreach \i/\v in {0/1,1/3}
      \node[font=\scriptsize] at (\i+0.5,0.5) {\v};
    \foreach \i/\v in {0/2,1/4}
      \node[font=\scriptsize] at (\i+0.5,-0.5) {\v};
  \end{tikzpicture}
  \,\right]
  = \gamma\!\left[\,
  \begin{tikzpicture}[baseline=-2.75pt, scale=0.45]
    \foreach \i in {0,1} {
      \draw (\i,0) rectangle (\i+1,1);
      \draw (\i,-1) rectangle (\i+1,0);
    }
    \foreach \i/\v in {0/1,1/2}
      \node[font=\scriptsize] at (\i+0.5,0.5) {\v};
    \foreach \i/\v in {0/3,1/4}
      \node[font=\scriptsize] at (\i+0.5,-0.5) {\v};
  \end{tikzpicture}
  \,\right]
  =\;
  \begin{tikzpicture}[baseline=-2.75pt, scale=0.45]
    \foreach \i in {0,1,2} {
      \draw (\i,0) rectangle (\i+1,1);
      \draw (\i,-1) rectangle (\i+1,0);
    }
    \foreach \i/\v in {0/1,1/2,2/5}
      \node[font=\scriptsize] at (\i+0.5,0.5) {\v};
    \foreach \i/\v in {0/3,1/4,2/6}
      \node[font=\scriptsize] at (\i+0.5,-0.5) {\v};
  \end{tikzpicture}\,.
\]

\subsection{Non-crossing partitions}\label{sec:ncp}

A non-crossing partition~\cite{Kreweras72} of $[n]$ is a
set partition whose blocks do not cross, that is, there are no
$a < b < c < d$ with $a, c$ in one block and $b, d$ in
another. An indecomposable
non-crossing partition is one in which $1$ and $n$ belong
to the same block. The decomposition below, in which
$\lambda$ adds a new maximum to the block of~$1$,
appears in~\cite[Proposition~9.4]{NicaSpeicher06}.
\begin{itemize}
\item $e$ is the empty partition.
\item $\lambda(P)$: add $n{+}1$ to the block
  containing~$1$. In particular,
  $\lambda(e) = \{\{1\}\}$.
\item $P_1 \oplus P_2 = P_1 \cup (P_2 + |P_1|)$:
  shift the elements of~$P_2$ by $|P_1|$ and take the
  union.
\item $\gamma(P)$: adjoin the singleton block
  $\{n{+}1\}$.
\item $P_1 \ominus P_2$: if $P_1$ is empty, $P_1 \ominus P_2 = P_2$.
  If $P_1$ is decomposable, let $V$
  be its last indecomposable summand and $U$ the rest, so
  $P_1 = U \oplus V$; then
  $P_1 \ominus P_2 = U \oplus (V \ominus P_2)$. If $P_1$
  is indecomposable, write $P_1 = \lambda(P_1')$ (remove
  $n$ from the block of~$1$, leaving a partition of
  $[n{-}1]$) and set
  $P_1 \ominus P_2 = \lambda(P_1' \ominus P_2)$, that is,\
  add the new maximum to the block of~$1$.
\end{itemize}
See Figure~\ref{fig:ncp-ops}.

\begin{figure}[ht]
\centering
\begin{tikzpicture}[
  dot/.style={circle, fill, inner sep=1.2pt},
  scale=0.75
]
% === P ===
\begin{scope}
  \foreach \i in {1,...,4}
    \node[dot] (p\i) at (\i, 0) {};
  \foreach \i in {1,...,4}
    \node[below=3pt, font=\scriptsize] at (\i, 0) {\i};
  \draw (p1) to[bend left=55] (p4);
  \draw (p2) to[bend left=70] (p3);
  \node[font=\small] at (2.5, -1.3) {$P$};
\end{scope}
% === lambda(P) ===
\begin{scope}[shift={(5.5, 0)}]
  \foreach \i in {1,...,5}
    \node[dot] (q\i) at (\i, 0) {};
  \foreach \i in {1,...,5}
    \node[below=3pt, font=\scriptsize] at (\i, 0) {\i};
  \draw (q1) to[bend left=55] (q4);
  \draw (q4) to[bend left=70] (q5);
  \draw (q2) to[bend left=70] (q3);
  \node[font=\small] at (3, -1.3) {$\lambda(P)$};
\end{scope}
% === gamma(P) ===
\begin{scope}[shift={(12, 0)}]
  \foreach \i in {1,...,5}
    \node[dot] (r\i) at (\i, 0) {};
  \foreach \i in {1,...,5}
    \node[below=3pt, font=\scriptsize] at (\i, 0) {\i};
  \draw (r1) to[bend left=55] (r4);
  \draw (r2) to[bend left=70] (r3);
  \node[font=\small] at (3, -1.3) {$\gamma(P)$};
\end{scope}
\end{tikzpicture}
\caption{The non-crossing partition
$P = \{\{1,4\},\{2,3\}\}$ with its blocks drawn as arcs,
and its images under $\lambda$ and $\gamma$:
$\lambda(P) = \{\{1,4,5\},\{2,3\}\}$ adds the new
maximum $5$ to the block containing $1$, while
$\gamma(P) = \{\{1,4\},\{2,3\},\{5\}\}$ adjoins $\{5\}$ as a new
singleton block.}
\label{fig:ncp-ops}
\end{figure}
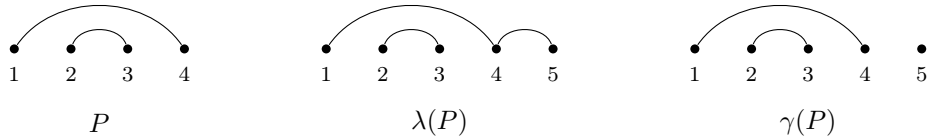

No closed form for $h$ on non-crossing partitions is
known; $h$ is computed via the generic last-return
recursion of Lemma~\ref{lem:h-last-return}.

\subsection{Concrete descriptions of the canonical
  bijections}\label{sec:bijections}

There are $\binom{8}{2} = 28$ canonical bijections
$\varphi_{A \to B}$ between our eight families, each
defined recursively via the abstract framework, so their
effect on a given object is not immediately clear.
For a spanning set of pairs we give concrete, usually
non-recursive descriptions. The seven pairs below connect
all eight families (they form a spanning tree of $K_8$),
so any $\varphi_{A \to B}$ is a composition of bijections
from this set, and these seven descriptions determine all
28. To verify that a concrete description matches the
canonical bijection, it suffices to check
by structural induction the following three identities.

\begin{proposition}
A bijection $f : A \to B$ between two Catalan structures
is the canonical bijection if and only if:
\begin{itemize}
\item $f(e_A) = e_B$,
\item $f(\lambda_A(t)) = \lambda_B(f(t))$, and
\item $f(u \oplus_A v) = f(u) \oplus_B f(v)$.
\end{itemize}
\end{proposition}

Since decomposition into $e$, $\lambda$ and $\oplus$ is
unique, these identities force
$f = \varphi_{A \to B}$. The verifications are carried out
in Appendix~\ref{sec:bij-verify}.

\paragraph{Triangulations $\to$ plane trees.}
Vertex $1$ of the $(n{+}2)$-gon becomes the root. The
vertices adjacent to $1$ (other than $n{+}2$) become the
children of the root; vertex $n{+}2$ is a dummy that
disappears. The construction recurses into each subpolygon
between consecutive neighbours of $1$. To reverse: label
the non-root vertices of the tree in preorder and recover
the triangulation. See
Figure~\ref{fig:tri-to-tree}.

\begin{figure}[ht]
\centering
\begin{tikzpicture}[
  nd/.style={circle, fill, inner sep=1.5pt},
  scale=0.75
]
% === Left: hexagon triangulation ===
\begin{scope}[shift={(-4, 0)}]
  \node[nd] (v1) at (2, 0) {};
  \node[nd] (v2) at (1, 1.732) {};
  \node[nd] (v3) at (-1, 1.732) {};
  \node[nd] (v4) at (-2, 0) {};
  \node[nd] (v5) at (-1, -1.732) {};
  \node[nd] (v6) at (1, -1.732) {};
  \draw (v1) -- (v2) -- (v3) -- (v4) -- (v5) -- (v6) -- (v1);
  \draw (v1) -- (v3);
  \draw (v3) -- (v5);
  \draw (v1) -- (v5);
  \node[right=2pt, font=\small] at (v1) {$1$};
  \node[above right, font=\small] at (v2) {$2$};
  \node[above left, font=\small] at (v3) {$3$};
  \node[left=2pt, font=\small] at (v4) {$4$};
  \node[below left, font=\small] at (v5) {$5$};
  \node[below right, font=\small] at (v6) {$6$};
\end{scope}
% === Arrow ===
\draw[->, thick] (-0.5, 0) -- (0.5, 0);
\node[above, font=\small] at (0, 0.05) {$f$};
% === Right: plane tree ===
\begin{scope}[shift={(3.3, 1.3)}]
  \node[nd] (r) at (0, 0) {};
  \node[nd] (c1) at (-1.4, -1.3) {};
  \node[nd] (c2) at (0, -1.3) {};
  \node[nd] (c3) at (1.4, -1.3) {};
  \node[nd] (g)  at (0, -2.6) {};
  \draw (r) -- (c1);
  \draw (r) -- (c2);
  \draw (r) -- (c3);
  \draw (c2) -- (g);
\end{scope}
\end{tikzpicture}
\caption{The triangulations-to-trees bijection. Vertex
$1$ becomes the root, and its neighbours $2, 3, 5$
(excluding the dummy vertex $n{+}2 = 6$) become its three
children. The three subpolygons recurse into the three
subtrees: each of the edges $(2,3)$ and $(5,6)$ yields a
leaf, while the triangle $(3,4,5)$ yields a subtree with
a single edge.}
\label{fig:tri-to-tree}
\end{figure}
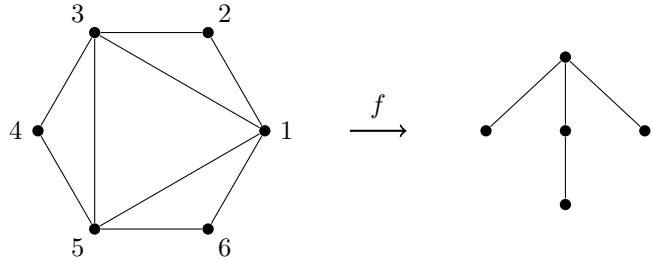

\paragraph{Plane trees $\to$ Dyck paths.}
Traverse the tree by depth-first search with children
visited left to right; each step down an edge produces a
$u$-step, each step back up produces a $d$-step.

\paragraph{Dyck paths $\to$ $2 \times n$ SYT.}
The $i$-th step of the Dyck path is $u$ if $i$ lies in
the top row of the tableau, and $d$ if $i$ lies in the
bottom row.

\paragraph{Plane trees $\to$ binary trees.}
This is the left-child/right-sibling encoding. If $T$
is a single node, $f(T) = \varnothing$. Otherwise, let
$\eta$ be the leftmost edge of $T$, let $T_1$ be the
subtree below $\eta$, and let $T_2$ be the plane tree
obtained from $T$ by deleting $T_1$ (so $T_2$ shares
its root with $T$). Then $f(T)$ is the binary tree with
root corresponding to $\eta$, left subtree $f(T_1)$,
and right subtree $f(T_2)$. See
Figure~\ref{fig:lc-rs}.

\begin{figure}[ht]
\centering
\begin{tikzpicture}[
  nd/.style={circle, fill, inner sep=1.5pt},
  scale=0.85
]
% === Left: plane tree ===
\begin{scope}[shift={(-3.2, 0)}]
  \node[nd] (r)  at (0, 0) {};
  \node[nd] (a)  at (-1.4, -1.2) {};
  \node[nd] (b)  at (0, -1.2) {};
  \node[nd] (c)  at (1.4, -1.2) {};
  \node[nd] (a1) at (-1.4, -2.4) {};
  \draw (r)  -- node[midway, above left, font=\scriptsize] {$\eta_1$} (a);
  \draw (r)  -- node[midway, below right=2pt, font=\scriptsize] {$\eta_3$} (b);
  \draw (r)  -- node[midway, above right, font=\scriptsize] {$\eta_4$} (c);
  \draw (a)  -- node[midway, left=2pt, font=\scriptsize] {$\eta_2$} (a1);
  \node[font=\small] at (0, -3.2) {$T$};
\end{scope}
% === Arrow ===
\draw[->, thick] (-0.6, -1) -- (0.6, -1);
\node[above, font=\small] at (0, -1) {$f$};
% === Right: binary tree ===
\begin{scope}[shift={(3.2, 0)}]
  \node[nd, label={right:$\eta_1$}] (B1) at (0, 0) {};
  \node[nd, label={left:$\eta_2$}]  (B2) at (-1, -1.2) {};
  \node[nd, label={right:$\eta_3$}] (B3) at (1, -1.2) {};
  \node[nd, label={right:$\eta_4$}] (B4) at (2, -2.4) {};
  \draw (B1) -- (B2);
  \draw (B1) -- (B3);
  \draw (B3) -- (B4);
  \node[font=\small] at (0, -3.2) {$f(T)$};
\end{scope}
\end{tikzpicture}
\caption{The left-child/right-sibling encoding. Each
edge $\eta_i$ of the plane tree~$T$ becomes a node of
the binary tree~$f(T)$: its left child in $f(T)$ is
the leftmost child edge of $\eta_i$'s lower endpoint
(if any), and its right child in $f(T)$ is the next
sibling edge of $\eta_i$ (if any). Here $T_1$ is the
subtree below $\eta_1$ (contributing $\eta_2$) and
$T_2$ is the rest of~$T$ sharing its root (contributing
$\eta_3, \eta_4$).}
\label{fig:lc-rs}
\end{figure}
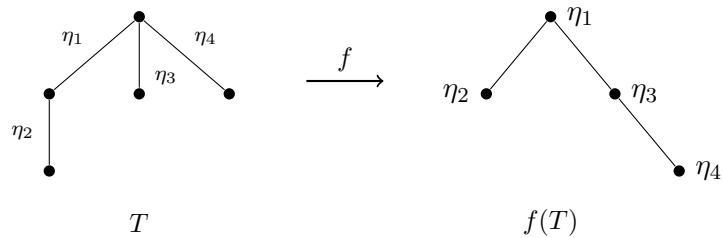

\paragraph{Plane trees $\to$ non-crossing
  partitions.}
Label the edges of $T$ in postorder with labels
$1, 2, \ldots, n$ (equivalently, label them
$n, n{-}1, \ldots, 1$ in the first-descent order of a
depth-first search with children visited right to left;
see Figure~\ref{fig:preorder-postorder}).
The labels along the leftmost path form the first block
of the partition. Remove these edges and recurse on the
resulting forest.

\paragraph{Plane trees $\to$ $231$-avoiding
  permutations.}
Label the edges of $T$ in postorder. Then read the labels
in preorder. The resulting sequence is a $231$-avoiding
permutation.

\paragraph{Plane trees $\to$ $321$-avoiding
  permutations.}
Label the edges of $T$ in postorder. Traverse $T$ by
depth-first search (children left to right). Going down
an edge not incident to a leaf: reserve an empty slot.
Going down an edge incident to a leaf: place its label
immediately. Going up a non-leaf edge: place its label in
the leftmost available slot. The result is a
$321$-avoiding permutation.

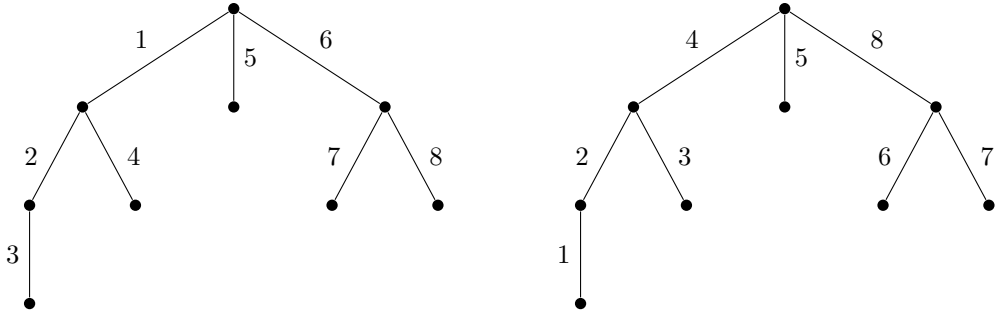
\begin{figure}[ht]
\centering
\begin{tikzpicture}[
  nd/.style={circle, fill, inner sep=1.5pt}
]
% preorder labeling
\node[nd] (r) at (3,0) {};
\node[nd] (a) at (1,-1.3) {};
\node[nd] (b) at (3,-1.3) {};
\node[nd] (c) at (5,-1.3) {};
\node[nd] (d) at (0.3,-2.6) {};
\node[nd] (e) at (1.7,-2.6) {};
\node[nd] (f) at (4.3,-2.6) {};
\node[nd] (g) at (5.7,-2.6) {};
\node[nd] (h) at (0.3,-3.9) {};
\draw (r) -- node[midway, above left, font=\small] {$1$} (a);
\draw (a) -- node[midway, left=3pt, font=\small] {$2$} (d);
\draw (d) -- node[midway, left, font=\small] {$3$} (h);
\draw (a) -- node[midway, right=3pt, font=\small] {$4$} (e);
\draw (r) -- node[midway, right, font=\small] {$5$} (b);
\draw (r) -- node[midway, above right, font=\small] {$6$} (c);
\draw (c) -- node[midway, left=3pt, font=\small] {$7$} (f);
\draw (c) -- node[midway, right=3pt, font=\small] {$8$} (g);
\end{tikzpicture}
\hspace{10mm}
\begin{tikzpicture}[
  nd/.style={circle, fill, inner sep=1.5pt}
]
% postorder labeling
\node[nd] (r) at (3,0) {};
\node[nd] (a) at (1,-1.3) {};
\node[nd] (b) at (3,-1.3) {};
\node[nd] (c) at (5,-1.3) {};
\node[nd] (d) at (0.3,-2.6) {};
\node[nd] (e) at (1.7,-2.6) {};
\node[nd] (f) at (4.3,-2.6) {};
\node[nd] (g) at (5.7,-2.6) {};
\node[nd] (h) at (0.3,-3.9) {};
\draw (r) -- node[midway, above left, font=\small] {$4$} (a);
\draw (a) -- node[midway, left=3pt, font=\small] {$2$} (d);
\draw (d) -- node[midway, left, font=\small] {$1$} (h);
\draw (a) -- node[midway, right=3pt, font=\small] {$3$} (e);
\draw (r) -- node[midway, right, font=\small] {$5$} (b);
\draw (r) -- node[midway, above right, font=\small] {$8$} (c);
\draw (c) -- node[midway, left=3pt, font=\small] {$6$} (f);
\draw (c) -- node[midway, right=3pt, font=\small] {$7$} (g);
\end{tikzpicture}
\caption{Edge labelings of a plane tree with eight edges.
Left: Preorder (labels assigned in the first-descent order
of a depth-first search with children visited left to right).
Right: Postorder (equivalently, labels assigned in
descending order $8, 7, \ldots, 1$ in the first-descent order
of a depth-first search with children visited right to left).
The constructions for non-crossing partitions, $\Av(231)$,
and $\Av(321)$ in Section~\ref{sec:bijections} use the postorder
labeling.}
\label{fig:preorder-postorder}
\end{figure}

\begin{remark}
Most of these concrete descriptions are classical
bijections: The DFS encoding of trees as Dyck paths,
the step/row correspondence between Dyck paths and SYT,
and the left-child/right-sibling encoding of plane trees
as binary trees are all standard. Non-crossing partitions
and $231$-avoiders appear as items (pp) and (ff) (via
reverse--complement) of Stanley's
list~\cite[Exercise~6.19]{Stanley99}. For non-crossing
partitions, a tree bijection via preorder vertex labeling
is given in the solution to item~159
of~\cite{Stanley15}; this is the Dershowitz--Zaks
construction~\cite{DershowitzZaks86}, and, while similar,
it differs from the postorder edge labeling used above.
The slot-filling procedure for $321$-avoiders appears
to be less standard.
\end{remark}

\section{Statistics and transport}\label{sec:statistics}

With the concrete families in hand, we turn to the
equidistribution theorem and the dictionary that
translates it into per-family statements.

\subsection{Canonical statistics and
  equidistribution}\label{sec:equidistribution}

The four \emph{canonical statistics} $\leaves$, $\subtrees$, $\internal$ and
$\rpath$ are defined by the following recursive equations on the free Catalan
structure $(C, e, \lambda, \oplus)$, where each $u \oplus v$ clause has both
summands nonempty:
\[
\begin{aligned}
  \leaves(e) &= 1 \\
  \leaves(\lambda(t)) &= \leaves(t) \\
  \leaves(u \oplus v)
    &= \leaves(u) + \leaves(v) \\[4pt]
  \internal(e) &= 0 \\
  \internal(\lambda(t))
    &= 1 + \internal(t) \\
  \internal(u \oplus v)
    &= \internal(u) + \internal(v) - 1
\end{aligned}
\qquad
\begin{aligned}
  \subtrees(e) &= 0 \\
  \subtrees(\lambda(t)) &= 1 \\
  \subtrees(u \oplus v)
    &= \subtrees(u) + \subtrees(v) \\[4pt]
  \rpath(e) &= 0 \\
  \rpath(\lambda(t))
    &= 1 + \rpath(t) \\
  \rpath(u \oplus v) &= \rpath(v).
\end{aligned}
\]

On plane trees, these have concrete descriptions:
$\leaves$ is the number of leaves,
$\internal$ the number of non-leaf vertices
($= |t| + 1 - \leaves(t)$ for $|t| \geq 1$),
$\subtrees$ the number of subtrees of the root
(that is, the number of children), and
$\rpath$ the length of the rightmost path.

\begin{theorem}[Equidistribution]\label{thm:equidistribution}
For all $t$ with $|t| \geq 1$,
\[
  \rpath(h(t)) = \subtrees(t)
  \quad\text{and}\quad
  \leaves(h(t)) = \internal(t).
\]
In particular, for $n \geq 1$, $(\subtrees, \leaves)$ and
$(\rpath, \internal)$ have the same joint distribution
on~$C_n$.
\end{theorem}

The equidistribution can be seen as a special case of Theorem~6 in
\cite{CKS13}. We give a self-contained proof.

\begin{proof}
It suffices to show $\rpath(h(t)) = \subtrees(t)$ for
all $t$, and $\leaves(h(t)) = \internal(t)$ for
$|t| \geq 1$. The identities $\subtrees(h(t)) = \rpath(t)$ and
$\internal(h(t)) = \leaves(t)$ then follow from
$h^2 = \operatorname{id}$. We use two auxiliary
identities, each proved by induction on $a$ using the
defining recurrence of $\ominus$:
\[
  \rpath(a \ominus b) = \rpath(a) + \rpath(b),
  \qquad
  \leaves(a \ominus b) + 1 = \leaves(a) + \leaves(b).
\]
We proceed by induction on $|t|$. The base cases are immediate, since $h$ fixes
both $e$ and $\lambda(e)$: the identity $\rpath(h(t)) = \subtrees(t)$ holds at
$t = e$ and at $t = \lambda(e)$, while $\leaves(h(t)) = \internal(t)$, asserted
only for $|t| \geq 1$, holds at $t = \lambda(e)$.

For the inductive step ($|t| \geq 2$), write $t$ in
first-return form.

\emph{Case 1: $t = \lambda(u)$ with $|u| \geq 1$.}
Since $h(u) \neq e$, we have
$h(t) = \gamma(h(u)) = h(u) \oplus \lambda(e)$, so
\begin{align*}
  \rpath(h(t))
    &= \rpath(\lambda(e)) = 1 = \subtrees(t), \\
  \leaves(h(t))
    &= \leaves(h(u)) + 1
    = \internal(u) + 1 = \internal(t),
\end{align*}
using the induction hypothesis on $u$ in the second
line.

\emph{Case 2: $t = \lambda(u) \oplus v$ with $v \neq e$.}
By the definition of $h$,
\[
  h(t) = h(v) \ominus \gamma(h(u))
       = h(v) \ominus h(\lambda(u)).
\]
Applying the auxiliary identities together with the
induction hypothesis on $v$ and on $\lambda(u)$ (both
of size ${<}\,|t|$ and ${\geq}\,1$),
\begin{align*}
  \rpath(h(t))
    &= \rpath(h(v)) + \rpath(h(\lambda(u))) \\
    &= \subtrees(v) + \subtrees(\lambda(u))
    = \subtrees(t), \\
  \leaves(h(t)) + 1
    &= \leaves(h(v)) + \leaves(h(\lambda(u))) \\
    &= \internal(v) + \internal(\lambda(u))
    = \internal(t) + 1,
\end{align*}
where the last equality uses
$\internal(u' \oplus v') = \internal(u') + \internal(v') - 1$
for nonempty $u', v'$.
\end{proof}

\subsection{Translating canonical statistics to native
  ones}\label{sec:dictionary}

With the equidistribution established at the abstract
level, it remains to identify the canonical statistics with
classical, natively defined statistics on each concrete
family. The full dictionary is as follows (the native statistics
are defined after the table):

\begin{center}
\renewcommand{\arraystretch}{1.3}
\begin{tabular}{llllllll}
Trees & Dyck & Tri & $\Av(231)$ & $\Av(321)$ &
  Binary & SYT & NCP \\
\hline
$\subtrees$ & $\comp$ & $\comp$ & $\comp$ & $\comp$ & $\comp$ & $\comp$
  & $\comp$ \\
$\rpath$ & $\rcomp$ & $\rcomp$ & $\rmax$ & $\rcomp$ & $\rcomp$
  & $\rcomp$ & $\rcomp$ \\
$\leaves$ & $\peak$ & $\tr$ & $\rmin$ & $\rmin$ & $\redge$
  & $\noncons$ & $\blocks$ \\
$\internal$ & $\nonpeak$ & $\nontr$ & $\nonrmin$ & $\nonrmin$
  & $\ledge$ & $\cons$ & $\nonmin$ \\
\end{tabular}
\end{center}

\emph{Convention:} Some statistics are shifted from
their naive combinatorial definition. For example, on
binary trees $\redge$ is one plus the number of right
edges, where a right edge is an edge from a node to a
nonempty right child. The
base-case caveat (Remark~\ref{rem:base-case}) applies to
the last two rows of the dictionary. The complement
statistic $n + 1 - \leaves$ is not independent but
appears on the right-hand side of the equidistribution,
so its native descriptions are included for each family.

\paragraph{Dyck paths.}
\begin{itemize}
\item comp = number of returns to the $x$-axis
\item rcomp = length of the terminal $d$-run
\item peak = number of occurrences of $ud$
\item nonpeak = $1$ + number of occurrences of $uu$
\end{itemize}

\paragraph{Triangulations of the $(n{+}2)$-gon.}
\begin{itemize}
\item comp = number of edges incident to $1$, minus $1$
\item rcomp = number of edges incident to $n{+}2$, minus $1$
\item tr = number of $i \in \{2, \ldots, n{+}1\}$ such
  that some $j < i$ is adjacent to both $i$ and $i{+}1$
\item nontr = $n + 1 - \tr$
\end{itemize}
Note that tr counts the triangles whose two largest-labeled vertices
are joined by a boundary edge.

\paragraph{$231$-avoiders.}
\begin{itemize}
\item comp = number of components (each begins
  with its max)
\item rmax = number of right-to-left maxima
\item rmin = number of right-to-left minima
\item nonrmin = $1$ + number of non-right-to-left-minima
\end{itemize}

\paragraph{$321$-avoiders.}
\begin{itemize}
\item comp = number of components
\item rcomp = $n - \pi(n) + 1$, where $\pi(n)$ is the last entry of $\pi$
\item rmin = number of right-to-left minima
\item nonrmin = $1$ + number of non-right-to-left-minima
\end{itemize}

\paragraph{Binary trees.}
\begin{itemize}
\item comp = number of nodes on the right spine
\item rcomp = number of right-spine peelings
  (see Section~\ref{sec:dict-binary})
\item redge = $1$ + number of right edges
\item ledge = $n + 1 - \redge$
\end{itemize}

\paragraph{$2 \times n$ SYT.}
\begin{itemize}
\item comp = number of components
\item rcomp = length of maximal consecutive segment at
  right end of bottom row
\item cons = $1$ + number of top-row entries $i$ such that
  $i{+}1$ is also in the top row
\item noncons = number of top-row entries $i$ such that
  $i{+}1$ is in the bottom row
\end{itemize}

\paragraph{Non-crossing partitions.}
\begin{itemize}
\item comp = number of components
\item rcomp = number of right-to-left maxima in flattened
  form
\item blocks = number of blocks
\item nonmin = $n + 1 - \blocks$
\end{itemize}

\subsection{Verifying the dictionary}\label{sec:transport}

Each entry in the dictionary of Section~\ref{sec:dictionary}
is a claim, namely that the
transported canonical statistic coincides with a
classical, natively defined statistic on the target family. For
instance, the entry $\text{leaves} \mapsto \text{peak}$
for Dyck paths asserts that
$\leaves(\varphi^{-1}(w)) = \peak(w)$ for every Dyck
path~$w$. These
identifications are where the main combinatorial work
lies; without them, the per-family equidistribution
statements in Section~\ref{sec:corollaries} are only formal rewritings.

The verification strategy is uniform: show that the
native statistic satisfies the same recursion (under $e$,
$\lambda$, $\oplus$) as the canonical statistic. The full
verification of every dictionary entry is carried out in
Appendix~\ref{sec:dict-verify}. As a model, we prove the
$\text{leaves} \mapsto \text{peak}$ entry in detail.

The canonical bijection
$\varphi : \text{Trees} \to \text{Dyck}$ is given by
\begin{align*}
  \varphi(e) &= \epsilon \\
  \varphi(\lambda(t)) &= u \varphi(t) d \\
  \varphi(s \oplus t)
    &= \varphi(s) \varphi(t).
\end{align*}
We must show that $\peak$ satisfies the same recursion
as $\leaves$ through this bijection, taking the two
operations in turn.

\emph{Concatenation.} If both $w_1$ and $w_2$ are
nonempty then $w_1$ ends with $d$ and $w_2$ starts
with $u$, so pairing the last letter of $w_1$ with
the first letter of $w_2$ yields $du$, not a peak;
if either word is empty the claim is trivial. Therefore
\[
  \peak(w_1 w_2)
    = \peak(w_1) + \peak(w_2),
\]
matching
$\leaves(s \oplus t) = \leaves(s) + \leaves(t)$.

\emph{Lifting.} When $w \neq \epsilon$, $w$ starts
with $u$ and ends with $d$, so
$\peak(uwd) = \peak(w)$, matching
$\leaves(\lambda(t)) = \leaves(t)$.
When $w = \epsilon$,
$\peak(ud) = 1 = \leaves(\lambda(e))$, completing the
induction for $n \geq 1$.

\begin{remark}[Base-case caveat]\label{rem:base-case}
All dictionary identifications for $\leaves$ and
$\internal$ hold for $n \geq 1$. At $n = 0$ the
native statistics may differ from the canonical values
$\leaves(e) = 1$ and $\internal(e) = 0$. For instance,
$\peak(\epsilon) = 0 \neq 1 = \leaves(e)$. This
discrepancy does not affect the equidistribution theorem,
which concerns $n \geq 1$.
\end{remark}

\subsection{Per-family
  corollaries}\label{sec:corollaries}

Combining the equidistribution theorem
(Theorem~\ref{thm:equidistribution}) with the dictionary
identifications from Section~\ref{sec:dictionary} gives the
following per-family
statements, each holding for $n \geq 1$. The involution
$h$ sends:
\begin{itemize}
\item $(\comp, \peak)$ to $(\rcomp, \nonpeak)$ on
  Dyck paths,
\item $(\comp, \tr)$ to $(\rcomp, \nontr)$ on
  triangulations,
\item $(\comp, \rmin)$ to $(\rmax, \nonrmin)$ on
  $\Av(231)$,
\item $(\comp, \rmin)$ to $(\rcomp, \nonrmin)$ on
  $\Av(321)$,
\item $(\comp, \redge)$ to $(\rcomp, \ledge)$ on
  binary trees,
\item $(\comp, \noncons)$ to $(\rcomp, \cons)$ on
  $2 \times n$ SYT, and
\item $(\comp, \blocks)$ to $(\rcomp, \nonmin)$ on
  non-crossing partitions.
\end{itemize}

The equidistribution of the individual statistics (for example, that
peaks on Dyck paths have the Narayana distribution~\cite{Deutsch99,Narayana55})
is classical.  What the present framework provides is the \emph{joint}
equidistribution via an explicit involution.  In \cite{CKS13} this joint result
is established on $\beta(1,0)$-trees, but the per-family statements
above, which require the dictionary identifications of Section~\ref{sec:dictionary}
to translate the canonical statistics into native ones, are new.

\subsection{Generating functions}\label{sec:genfun}

The four canonical statistics admit a joint
generating function, which we now derive.
Define
\[
  F = F(t,u) = \sum_{T \in C} t^{|T|}\,
    u^{\leaves(T)}.
\]
The empty tree contributes $u$. Every nonempty tree
is either indecomposable, $T = \lambda(A)$, or decomposable,
$T = \lambda(A) \oplus B$ with $B \neq e$.
In the indecomposable case $\leaves(T) = \leaves(A)$ and
$|T| = 1 + |A|$, so these trees contribute $tF$.
In the decomposable case $\leaves(T) = \leaves(A) + \leaves(B)$
and $|T| = 1 + |A| + |B|$, and the sums over $A$ (all
trees) and $B$ (nonempty trees) factor, contributing
$tF(F - u)$. Therefore
$F - u = tF + tF(F - u)$, that is,\
\[
  tF^2 + (t - tu - 1)F + u = 0.
\]
Writing $\Delta = (1-t+tu)^2 - 4tu$ for the
discriminant,
\[
  F(t,u)
    = \frac{1 - t + tu - \sqrt{\Delta}}{2t}.
\]
For $n \geq 1$ and $1 \leq k \leq n$, the coefficient
to $t^n\, u^k$ in $F$ is the Narayana number $N(n,k)$.
Specializing $u = 1$ gives $tF^2 - F + 1 = 0$, and hence
\[
  F(t,1) = \frac{1 - \sqrt{1 - 4t}}{2t}
    = \sum_{n \geq 0} |C_n|\, t^n,
\]
the Catalan generating function.

Introducing additional variables $p$ for $\subtrees$ and $q$
for $\rpath$, define
\[
  G = G(t,p,q,u) = \sum_{T \in C} t^{|T|}\,
    p^{\subtrees(T)}\, q^{\rpath(T)}\,
    u^{\leaves(T)}.
\]

\begin{theorem}[Generating function]\label{thm:genfun}
With $F = F(t,u)$ as above,
\[
  G(t,p,q,u) = u +
    \frac{tpqu\,(1-tF)}{(1-tpF)(1-tF-tq)}.
\]
Equivalently, with
$\Delta = (1-t+tu)^2 - 4tu$,
\[
  G(t,p,q,u) = u +
    \frac{tpqu\,(1+t-tu+\sqrt{\Delta})}
    {2(1-tq-tpu)
      + (1-t+tu-\sqrt{\Delta})(tp(u+q-1)-1)}.
\]
\end{theorem}

\begin{proof}
The empty tree contributes $u$ to $G$. For each
nonempty tree $T$, apply its first-return decomposition.

\emph{Indecomposable case} ($T = \lambda(A)$):
$\subtrees(T) = 1$,
$\rpath(T) = 1 + \rpath(A)$,
$\leaves(T) = \leaves(A)$, $|T| = 1 + |A|$.
Summing over $A$ gives $tpq\, G(t,1,q,u)$, since the
sum tracks $\rpath$ and $\leaves$ but not $\subtrees$.

\emph{Decomposable case}
($T = \lambda(A) \oplus B$, $B \neq e$):
$\subtrees(T) = 1 + \subtrees(B)$,
$\rpath(T) = \rpath(B)$,
$\leaves(T) = \leaves(A) + \leaves(B)$,
$|T| = 1 + |A| + |B|$.
The sum over $A$ contributes $F$ (only $t$ and $u$
are involved), the sum over nonempty $B$ contributes
$G - u$, the extra $t$ from $|T| = 1 + |A| + |B|$
contributes a factor $t$, and
$p^{1 + \subtrees(B)} = p \cdot p^{\subtrees(B)}$
introduces a factor $p$, giving $tpF(G - u)$ in total.

The two cases together yield
\[
  G - u = tpq\, G(t,1,q,u) + tpF(G - u),
\]
and hence
\[
  G = u + \frac{tpq\, G(t,1,q,u)}{1 - tpF}.
\]
When $p = 1$ this becomes
\[
  G(t,1,q,u) = u +
    \frac{tq\, G(t,1,q,u)}{1 - tF},
\]
so $G(t,1,q,u) = u(1 - tF)/(1 - tF - tq)$.
Feeding this back gives
\[
  G
  =  u + \frac{tpq}{1 - tpF}\,
    \frac{u(1 - tF)}{1 - tF - tq}
  = u + \frac{tpqu\,(1 - tF)}
    {(1 - tpF)(1 - tF - tq)}.
\]
For the explicit form, expand the denominator
$D = (1 - tpF)(1 - tF - tq)$:
\[
  D = 1 - tF - tq - tpF + t^2 pF^2 + t^2 pqF.
\]
The quadratic relation gives
$tF^2 = (1 - t + tu)\,F - u$, so
\[
  D = (1 - tq - tpu) + tF\,(tp(u + q - 1) - 1).
\]
Substituting
$tF = \tfrac{1}{2}(1 - t + tu - \sqrt{\Delta})$ and
$1 - tF = \tfrac{1}{2}(1 + t - tu + \sqrt{\Delta})$
yields the second form.
\end{proof}

Since the canonical bijections preserve all four
statistics, the generating function $G$ is the same for
every concrete Catalan family.

\section{Donaghey's map $M$ and iterated secondary
  structures}\label{sec:donaghey-section}

The involution $h$ can be factored as a composition of two simpler
involutions. This factorization underlies the connection to Donaghey's
automorphism on plane trees and the period theorem for iterated secondary
structures.

\subsection{The involutions $\rev$ and
  $\corev$}\label{sec:rev-corev}

On binary trees (see Section~\ref{sec:binary}), $\rev$ reverses
the sequence of left subtrees along the right spine,
while $\corev$ swaps left and right subtrees at every
node; both recurse into every subtree.
On the free Catalan structure we define $\rev$ and
$\corev$ by
\begin{align*}
  \rev(e) &= e \\
  \rev(\lambda(u) \oplus v) &= \rev(v) \oplus \lambda(\rev(u)) \\
\shortintertext{and}
  \corev(e) &= e \\
  \corev(\lambda(u) \oplus v) &= \lambda(\corev(v)) \oplus \corev(u).
\end{align*}

Note that setting $v=e$ above shows that $\rev$ commutes with $\lambda$.

\begin{lemma}\label{lem:rev-switch}
For all $u, v \in C$,
\[
  \rev(u \oplus \lambda(v))
    = \lambda(\rev(v)) \oplus \rev(u).
\]
\end{lemma}

\begin{proof}
By induction on the size of $u$.
If $u = e$ then both sides equal $\lambda(\rev(v))$.
If $u \neq e$, write $u = \lambda(a) \oplus b$. Then
\begin{align*}
  \rev(u \oplus \lambda(v))
    &= \rev(\lambda(a)
       \oplus (b \oplus \lambda(v))) \\
    &= \rev(b \oplus \lambda(v))
       \oplus \lambda(\rev(a)) \\
    &= (\lambda(\rev(v)) \oplus \rev(b))
       \oplus \lambda(\rev(a)) \\
    &= \lambda(\rev(v))
       \oplus (\rev(b)
       \oplus \lambda(\rev(a))) \\
    &= \lambda(\rev(v))
       \oplus \rev(u)
\end{align*}
where the third step uses the induction hypothesis on
$b$, and the last step folds the definition of $\rev$.
\end{proof}

\begin{corollary}
$\rev$ and $\corev$ are involutions.
\end{corollary}

\begin{proof}
  We treat $\rev$ first. The base case $\rev(\rev(e)) = e$ is immediate. For the
  inductive step, write $t = \lambda(u) \oplus v$ in first-return form, where
  $v$ may be $e$. Then
  $\rev(t) = \rev(v) \oplus \lambda(\rev(u))$, and by
  Lemma~\ref{lem:rev-switch},
\[
\rev(\rev(t))
  = \lambda(\rev(\rev(u))) \oplus \rev(\rev(v))
  = \lambda(u) \oplus v = t.
\]
The argument for $\corev$ is similar:
\begin{align*}
\corev(\corev(\lambda(u) \oplus v))
  &= \corev(\lambda(\corev(v)) \oplus \corev(u)) \\
  &= \lambda(\corev(\corev(u))) \oplus \corev(\corev(v)) \\
  &= \lambda(u) \oplus v. \qedhere
\end{align*}
\end{proof}

\begin{theorem}\label{thm:h-factorization}
$h = \rev \circ \corev \circ \rev$.
\end{theorem}

\begin{proof}
Set $F = \rev \circ \corev \circ \rev$. Then
$F(e) = e$. For $t = u \oplus \lambda(v)$:
\begin{align*}
  F(u \oplus \lambda(v))
    &= \rev(\corev(
         \rev(u \oplus \lambda(v)))) \\
    &= \rev(\corev(
         \lambda(\rev(v))
           \oplus \rev(u)))
      &\quad&[\text{Lemma~\ref{lem:rev-switch}}] \\
    &= \rev(
         \lambda(\corev(\rev(u)))
         \oplus \corev(\rev(v)))
      &&[\corev\text{ def.}] \\
    &= \rev(\corev(\rev(v)))
       \oplus
       \lambda(\rev(\corev(
         \rev(u))))
      &&[\rev\text{ def.}] \\
    &= F(v) \oplus \lambda(F(u)).
\end{align*}
So $F$ satisfies $F(e) = e$ and
$F(u \oplus \lambda(v)) = F(v) \oplus \lambda(F(u))$.
Since the last-return decomposition is unique, these two
equations determine $F$ by structural induction, and $h$
satisfies the same recursion
(Lemma~\ref{lem:h-last-return}), so $F = h$.
\end{proof}

\begin{remark}
Computing $h(t)$ from its defining recursion takes time $O(n^2)$ in the size
$n = |t|$, since each application of $\ominus$ traverses a rightmost path. The
factorization $h = \rev \circ \corev \circ \rev$ does better. Represent the
structure as a binary tree, each node holding pointers to its two subtrees. Then
$\corev$ swaps the two pointers at every node, and $\rev$ reverses the right-spine
list and recurses into every subtree. Each is a single traversal of the tree,
hence $O(n)$, and so is $h$.
\end{remark}

\subsection{Donaghey's map $M$}\label{sec:donaghey}

Since $h = \rev \circ \corev \circ \rev$
(Theorem~\ref{thm:h-factorization}), define
$M = h \circ \rev = \rev \circ \corev$.
This map turns out to \emph{intertwine} the primary and secondary
structures, meaning that $M$ conjugates $\lambda$ into
$\gamma$ and $\oplus$ into~$\ominus$.
Since these properties uniquely determine a
size-preserving map, $M$ must coincide with any other
map that intertwines the two structures, in particular
with Donaghey's automorphism~\cite{Donaghey80}, as we
verify below.

\begin{theorem}\label{thm:M-intertwine}
\begin{align*}
  M(e) &= e \\
  M(\lambda(t)) &= \gamma(M(t)) \\
  M(u \oplus v) &= M(u) \ominus M(v)
\end{align*}
\end{theorem}

\begin{proof}
The first equation is immediate. For the second,
since $\rev$ commutes with $\lambda$, we have
\[
  M(\lambda(t))
    = h(\rev(\lambda(t)))
    = h(\lambda(\rev(t)))
    = \gamma(h(\rev(t)))
    = \gamma(M(t)).
\]
For the third, we first note that for any
$u = \lambda(a) \oplus b$,
\begin{equation}\label{eq:M-first-return}
  M(u) = M(a) \oplus \lambda(M(b)).
\end{equation}
Indeed,
\begin{align*}
  M(\lambda(a) \oplus b)
    &= h(\rev(b) \oplus \lambda(\rev(a)))
      &\quad&[\rev\text{ def.}] \\
    &= h(\rev(a)) \oplus \lambda(h(\rev(b)))
      &&[h\text{ last-return}] \\
    &= M(a) \oplus \lambda(M(b)).
\end{align*}
Now we induct on the size of $u$. The base case $u = e$ is
trivial. For the inductive step, write $u = \lambda(a) \oplus b$, so
that $u \oplus v = \lambda(a) \oplus (b \oplus v)$. Then
\begin{align*}
  M(u \oplus v)
    &= M(a) \oplus \lambda(M(b \oplus v))
      &\quad&[\eqref{eq:M-first-return}] \\
    &= M(a) \oplus \lambda(M(b) \ominus M(v))
      &&[\text{ind.\ hyp.}] \\
    &= (M(a) \oplus \lambda(M(b))) \ominus M(v)
      &&[\text{Lemma~\ref{lem:ominus-snoc}}] \\
    &= M(u) \ominus M(v).
      &&[\eqref{eq:M-first-return}]
\end{align*}
This concludes the proof.
\end{proof}

\begin{corollary}
$M^{-1} = \rev \circ M \circ \rev$.  In particular, $M$ is conjugate to its own inverse by
$\rev$.

\end{corollary}

\begin{proof}
Since $\rev$ and $h$ are involutions,
\[
  M^{-1} = (h \circ \rev)^{-1} = \rev \circ h
    = \rev \circ h \circ \rev \circ \rev = \rev \circ M \circ \rev. \qedhere
\]
\end{proof}

\begin{theorem}\label{thm:M-no-fixedpoints}
The only fixed points of $M$ are $e$ and $\lambda(e)$.
\end{theorem}

\begin{proof}
Both $e$ and $\lambda(e)$ are fixed by $M$
(since $\rev$ and $h$ each fix them).
For the converse,
since $M = h \circ \rev$, $M(t) = t$ implies $h(t) = \rev(t)$.
For each nonempty element
$t = \lambda(s_1) \oplus \cdots \oplus \lambda(s_k)$,
define $\mathrm{first}(t) = 1 + |s_1|$ and
$\mathrm{last}(t) = 1 + |s_k|$ (the sizes of the first
and last indecomposable summands).

We claim that $\mathrm{last}(h(t))
= |t| - \mathrm{last}(t) + 1$
for all nonempty $t$. If $t = \lambda(a)$, then
$h(t) = \gamma(h(a))$ and the last summand of
$\gamma(h(a))$ is $\lambda(e)$, giving
$\mathrm{last}(h(t)) = 1 = |t| - |t| + 1$.
If $t = \lambda(a) \oplus v$ with $v$ nonempty, then
$h(t) = h(v) \ominus \gamma(h(a))$, which replaces the
leaf at the end of the rightmost-child chain in the last
summand of $h(v)$ by $\gamma(h(a))$, increasing that
summand's size by $|\gamma(h(a))| = |a| + 1$.
By induction,
$\mathrm{last}(h(t))
= \mathrm{last}(h(v)) + |a| + 1
= (|v| - \mathrm{last}(v) + 1) + |a| + 1
= |t| - \mathrm{last}(t) + 1$.

Similarly, $\mathrm{last}(\rev(t))
= \mathrm{first}(t)$
since $\rev$ reverses the sequence of indecomposable
summands (recursing into each).

Now if $h(t) = \rev(t)$ for nonempty $t$, applying
$\mathrm{last}$ to both sides gives
$|t| - \mathrm{last}(t) + 1 = \mathrm{first}(t)$,
that is,
$|t| + 1 = \mathrm{first}(t) + \mathrm{last}(t)$.
We consider three cases. For $k = 1$:
$\mathrm{first}(t) = \mathrm{last}(t) = |t|$, giving
$|t| + 1 = 2|t|$, so $|t| = 1$.
For $k = 2$: $|t| = \mathrm{first}(t) +
\mathrm{last}(t)$ (no middle summands), giving
$|t| + 1 = |t|$, a contradiction.
For $k \geq 3$: the middle summands contribute at
least $k - 2 \geq 1$ to $|t|$, so
$|t| > \mathrm{first}(t) + \mathrm{last}(t)$,
also a contradiction.
\end{proof}

We now describe $M$ concretely on plane trees.
The map $\rev$ reverses the child list at
every level:
\[
  \rev([c_1, \ldots, c_k])
    = [\rev(c_k), \ldots, \rev(c_1)].
\]
Using \eqref{eq:M-first-return} from the proof of Theorem~\ref{thm:M-intertwine}, $M$ admits the
recursive description
\begin{align}\label{eq:M-tree}
  M([\,]) &= [\,] \notag\\
  M([c_1, \ldots, c_k])
    &= M(c_1) \oplus \lambda(M([c_2, \ldots, c_k]))
\end{align}
since $[c_1, \ldots, c_k] = \lambda(c_1) \oplus
[c_2, \ldots, c_k]$ and
\eqref{eq:M-first-return} gives
$M(\lambda(c_1) \oplus [c_2, \ldots, c_k])
= M(c_1) \oplus \lambda(M([c_2, \ldots, c_k]))$.
Compare with the description of $h$ on trees
in Section~\ref{sec:trees}:
\[
  h([c_1, \ldots, c_k])
    = h(c_k) \oplus \lambda(h(c_{k-1}) \oplus
      \lambda(\cdots \oplus \lambda(h(c_1)
      \oplus \lambda(e)) \cdots)).
\]
The two have the same form, but $M$ processes children
left-to-right and $h$ right-to-left, reflecting
$h = M \circ \rev$.

\begin{theorem}[Donaghey's automorphism]\label{thm:donaghey}
The recursion \eqref{eq:M-tree} coincides with
Donaghey's automorphism~\cite{Donaghey80} on plane
trees.
\end{theorem}

\begin{proof}
Donaghey constructs his automorphism as $L^{-1} \circ R$,
where $R$ and $L$ are two maps between what he calls
``general bracketings'' and ``plus-binary bracketings''
(well-matched parenthesizations, and a variant in which
some brackets are replaced by $+$ signs). We describe
the construction in our language.

Both types of bracketing are counted by the Catalan
numbers and can be represented as plane trees. A general
bracketing corresponds to a plane tree in the usual way
(nested parentheses encode the parent--child relation).
A plus-binary bracketing has a right-nested structure:
its tree form is a (possibly empty) right-nested list
$[b_1, \ldots, b_k]$, where each $b_i$ is itself a
plus-binary bracketing.

The two representations correspond to different
decompositions in the free Catalan structure.
A general bracketing (plane tree) with child list
$[c_1, \ldots, c_k]$ decodes via the first-return
decomposition:
\[
  \lambda(c_1) \oplus \cdots \oplus \lambda(c_k).
\]
A plus-binary bracketing $[b_1, \ldots, b_k]$ decodes
by nesting each successive element inside the last
summand:
\[
  b_1 \oplus \lambda(b_2 \oplus
    \lambda(\cdots \oplus \lambda(b_k \oplus
    \lambda(e)) \cdots)).
\]
Figure~\ref{fig:donaghey-decoding} illustrates the two
decodings for $k = 3$.

\begin{figure}[ht]
\centering
\begin{tikzpicture}[
  nd/.style={circle, draw, inner sep=1.5pt,
    minimum size=14pt, font=\small},
  level distance=10mm, sibling distance=20mm,
  edge from parent/.style={draw},
  scale=0.85, transform shape
]
% Left: first-return
\node[font=\small] at (0, 1.2) {First-return};
\node[nd] (r1) at (0,0) {$\oplus$}
  child {node[nd] {$\lambda$}
    child {node[nd] {$c_1$}}
  }
  child {node[nd] {$\oplus$}
    child {node[nd] {$\lambda$}
      child {node[nd] {$c_2$}}
    }
    child {node[nd] {$\lambda$}
      child {node[nd] {$c_3$}}
    }
  };

% Right: right-nested
\begin{scope}[xshift=65mm]
\node[font=\small] at (0, 1.2) {Right-nested};
\node[nd] (r2) at (0,0) {$\oplus$}
  child {node[nd] {$c_1$}}
  child {node[nd] {$\lambda$}
    child {node[nd] {$\oplus$}
      child {node[nd] {$c_2$}}
      child {node[nd] {$\lambda$}
        child {node[nd] {$\oplus$}
          child {node[nd] {$c_3$}}
          child {node[nd] {$\lambda$}
            child {node[nd] {$e$}}
          }
        }
      }
    }
  };
\end{scope}
\end{tikzpicture}
\caption{The two decodings of a tree with children
$c_1, c_2, c_3$. The first-return decomposition
(left) produces
$\lambda(c_1) \oplus \lambda(c_2) \oplus \lambda(c_3)$.
The right-nested decoding (right) produces
$c_1 \oplus \lambda(c_2 \oplus
\lambda(c_3 \oplus \lambda(e)))$.
Donaghey's map $M$ reinterprets the first as
the second.}
\label{fig:donaghey-decoding}
\end{figure}
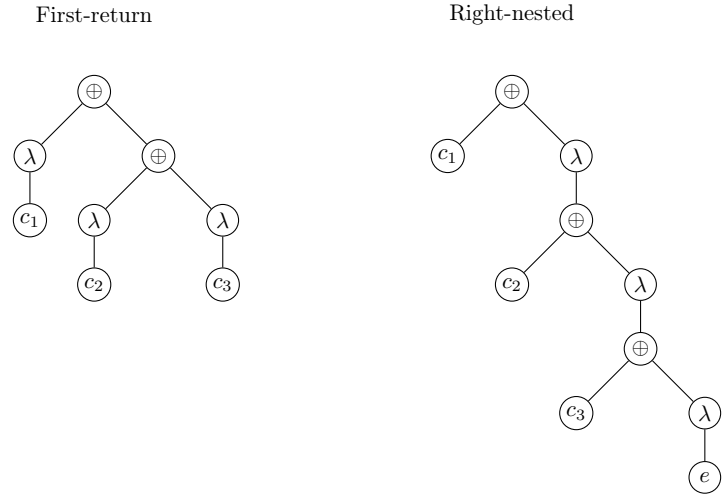

Donaghey's map $R$ (his Property~1) sends the general
bracketing $[c_1, \ldots, c_k]$ to the plus-binary
bracketing
$[R(c_1), \ldots, R(c_k)]$: it preserves the tree
shape but reinterprets the children as a right-nested
list. The inverse $L^{-1}$ converts a plus-binary
bracketing back to a general bracketing.
Thus $L^{-1} \circ R$ takes a plane tree, keeps its
shape, but changes the decoding from the first-return
pattern to the right-nested pattern, which is exactly
the recursion~\eqref{eq:M-tree}. The verification is
by induction on size: for $T = [\,]$ both give $e$;
for $T = [c_1, \ldots, c_k]$ with $k \geq 1$, writing
$T = \lambda(c_1) \oplus [c_2, \ldots, c_k]$ and
applying \eqref{eq:M-first-return} reduces to the
induction hypothesis on~$c_1$ and $[c_2, \ldots, c_k]$.
\end{proof}

\subsection{The free product structure}\label{sec:free-product}

Pushkarev and Byzov~\cite{PushkarevByzov14,PushkarevByzov19} independently
decompose Donaghey's transformation into two involutions, working on \emph{plane
  planted cubic trees} (PPCTs), that is, rooted plane trees in which every
non-leaf has exactly two children (a left son and a right son). PPCTs with $n$
leaves are in bijection with plane trees with $n - 1$ edges (equivalently, with
elements of~$C_{n-1}$).

Their two involutions are:
\begin{itemize}
\item $\sigma$, the \emph{transposition}: swap the left
  and right son at every non-leaf of a PPCT.
\item $\theta$, \emph{Dyck-path reflection}: given a
  Dyck path $w = w_1 \cdots w_{2n}$, reverse it to
  $w_{2n} \cdots w_1$ and exchange $u \leftrightarrow d$.
  This acts on PPCTs via the canonical bijection
  $\varphi$ from PPCTs to Dyck paths:
  $T \mapsto \varphi^{-1}(\theta(\varphi(T)))$.
\end{itemize}
Donaghey's transformation is then
$\varphi^{-1} \circ \theta \circ \varphi \circ
\sigma$~\cite[Definition~2.3]{PushkarevByzov14}.

\begin{proposition}\label{prop:pb-identification}
Under the left-child/right-sibling encoding,
$\sigma$ corresponds to~$\corev$ and
$\varphi^{-1} \circ \theta \circ \varphi$ corresponds
to~$\rev$.
\end{proposition}

\begin{proof}
In the first-return decomposition
$t = \lambda(u) \oplus v$, the left-child/right-sibling
encoding maps $u$ to the left subtree and $v$ to the
right subtree. Swapping left and right recursively
therefore sends $\lambda(u) \oplus v$ to
$\lambda(\corev(v)) \oplus \corev(u)$, which is the
defining recursion of~$\corev$. Thus $\sigma$
corresponds to~$\corev$.
For~$\theta$: by
Theorem~\ref{thm:h-factorization},
$M = h \circ \rev
= (\rev \circ \corev \circ \rev) \circ \rev
= \rev \circ \corev$.
Since
$\varphi^{-1} \circ \theta \circ \varphi \circ \sigma$
also equals $M$
(\cite[Theorem~2.1]{PushkarevByzov14} and
Theorem~\ref{thm:donaghey}), and $\sigma$
corresponds to~$\corev$, we get
$\varphi^{-1} \circ \theta \circ \varphi
= M \circ \corev
= (\rev \circ \corev) \circ \corev = \rev$.
\end{proof}

Shapiro~\cite{Shapiro79} proved
$M^6 = \operatorname{id}$ on compositions and posed the
order question in general. Pushkarev and Byzov
established the following stronger structural result.

\begin{theorem}[{\cite[Theorem~6.4]{PushkarevByzov14}}]
The group generated by $\sigma$ and
$\varphi^{-1} \circ \theta \circ \varphi$
acting on~$C$ is the free product
$\mathbb{Z}_2 * \mathbb{Z}_2$.
\end{theorem}

\begin{corollary}
$\langle h, \rev \rangle$, the group generated by $h$ and $\rev$,
equals $\mathbb{Z}_2 * \mathbb{Z}_2$.
In particular, $M = h \circ \rev$ has infinite order
and $\operatorname{ord}(M|_{C_n})$ is unbounded.
\end{corollary}

\begin{proof}
By Proposition~\ref{prop:pb-identification},
$\langle \sigma, \varphi^{-1} \circ \theta \circ
\varphi \rangle
= \langle \corev, \rev \rangle$. Since
$h = \rev \circ \corev \circ \rev$
(Theorem~\ref{thm:h-factorization}),
$\langle \corev, \rev \rangle
= \langle h, \rev \rangle$. By the theorem above this group is the
free product $\mathbb{Z}_2 * \mathbb{Z}_2$, whose only relations are
$h^2 = \rev^2 = \operatorname{id}$. Hence $M = h \circ \rev$ has
infinite order, and no single~$k$ satisfies $M^k = \operatorname{id}$
on all of~$C$. As each $M|_{C_n}$ permutes the finite set~$C_n$, its order is
finite; were these orders bounded, their least common multiple would be such a
global~$k$. Hence $\operatorname{ord}(M|_{C_n})$ is unbounded.
\end{proof}

\begin{remark}
  The group $\langle\rev,\corev\rangle$ is the infinite dihedral group
  $\mathbb{Z}_2*\mathbb{Z}_2$, whose reflections are the odd-length
  words in $\rev$ and $\corev$. Besides the generators, the two shortest
  are $h=\rev\circ\corev\circ\rev$ and
  $k:=\corev\circ\rev\circ\corev$. Since $\corev$ is an involution,
  $k=\corev\circ\rev\circ\corev^{-1}$ is conjugate to $\rev$; dually $h$
  is conjugate to $\corev$. The reflections of this group form exactly
  two conjugacy classes, represented by $\rev$ and $\corev$; hence $k$
  and $h$ lie in different classes and are not conjugate to each other.
  Since conjugate involutions have
  the same cycle type, $k$ and $\rev$ have equally many fixed points
  on each $C_n$. In the Dyck model $\rev$ is reversing the
  path, so its fixed points are the paths symmetric about their vertical
  midline, and
  \[
    |\{t \in C_n : k(t) = t\}|
    = |\{t \in C_n : \rev(t) = t\}|
    = \binom{n}{\lfloor n/2\rfloor},
  \]
  the central binomial coefficient (\oeis\ A001405). By contrast, $h$
  has $|C_m|$ fixed points of odd size $n = 2m+1$ and none of positive
  even size (Proposition~\ref{prop:h-fixed-points}).
\end{remark}

In~\cite{PushkarevByzov19}, Pushkarev and Byzov prove
that the number of orbits of~$M$ on~$C_n$ grows as
$\Theta(2^n)$ and construct explicit families of cycles
of lengths~$6$ and~$9$.
The order of~$M$ controls the periodicity of iterated
secondary structures.

\subsection{Iterated secondary structures}\label{sec:iterated}

Starting from $(\lambda_1, \oplus_1)$, the
secondary structure construction yields
$(\lambda_2, \oplus_2) = (\gamma, \ominus)$. Iterating
gives a sequence
\[
  (\lambda_1, \oplus_1) \to (\lambda_2, \oplus_2)
    \to (\lambda_3, \oplus_3) \to \cdots.
\]
The general step defines $\lambda_{i+1}$ by
\[
    \lambda_{i+1}(t) = t \oplus_i \lambda_i(e)
\]
and $\oplus_{i+1}$ recursively in the first argument by
\begin{align*}
           e \oplus_{i+1} v &= v \\
       \lambda_i(u) \oplus_{i+1} v
         &= \lambda_i(u \oplus_{i+1} v) \\
    (t \oplus_i u) \oplus_{i+1} v
      &= t \oplus_i (u \oplus_{i+1} v).
\end{align*}
The dual relations are
$\lambda_i(t) = \lambda_{i+1}(e) \oplus_{i+1} t$ and
\begin{align*}
    t \oplus_i e &= t \\
    t \oplus_i \lambda_{i+1}(u)
      &= \lambda_{i+1}(t \oplus_i u) \\
    t \oplus_i (u \oplus_{i+1} v)
      &= (t \oplus_i u) \oplus_{i+1} v \qquad (u \neq e).
\end{align*}

Since $h$ and $\rev$ both preserve size,
$M = h \circ \rev$ restricts to a permutation of each
finite set $C_k$.

\begin{theorem}\label{thm:period}
$(\lambda_{p+1}, \oplus_{p+1})
= (\lambda_1, \oplus_1)$ if and only if
$M^p = \operatorname{id}$ on $C$.
\end{theorem}

\begin{proof}
The intertwining theorem (Theorem~\ref{thm:M-intertwine}) gives
$M \circ \lambda_1 = \lambda_2 \circ M$ and
$M(u \oplus_1 v) = M(u) \oplus_2 M(v)$.
Induction on $i$ yields
$\lambda_{i+1} = M^i \circ \lambda_1 \circ M^{-i}$
and
$u \oplus_{i+1} v
= M^i(M^{-i}(u) \oplus_1 M^{-i}(v))$.
Indeed, the secondary structure construction
sends $(\lambda_i, \oplus_i)$ to
$(\lambda_{i+1}, \oplus_{i+1})$ using only
$\lambda_i$, $\oplus_i$, and $e$
(see Definition~\ref{def:secondary}); conjugating by $M$
preserves these relations, giving
$\gamma_i = M^{i-1} \circ \gamma_1 \circ M^{-(i-1)}$
and $\ominus_i$ transforms accordingly.

Now $(\lambda_{p+1}, \oplus_{p+1})
= (\lambda_1, \oplus_1)$ means
$M^p \circ \lambda_1 = \lambda_1 \circ M^p$ and
$M^p(u \oplus_1 v) = M^p(u) \oplus_1 M^p(v)$. Setting
$\phi = M^p$, these two conditions give
$\phi(\lambda(t)) = \lambda(\phi(t))$ and
$\phi(u \oplus v) = \phi(u) \oplus \phi(v)$. Together
with $\phi(e) = e$ (since $\phi$ preserves size), $\phi$
preserves the three-form decomposition. Since every
element is uniquely built from $e$, $\lambda$, and
$\oplus$, any size-preserving endomorphism that commutes
with these operations is the identity.
Conversely, $M^p = \operatorname{id}$
trivially implies
$(\lambda_{p+1}, \oplus_{p+1})
= (\lambda_1, \oplus_1)$.
\end{proof}

The smallest $p$ such that $M^p$ is the identity on $\bigsqcup_{k \leq n} C_k$
is the least common multiple
$\operatorname{lcm}_{0 \leq k \leq n} \operatorname{ord}(M|_{C_k})$.  The first
few values are (\oeis\ A060114):

\begin{center}
\begin{tabular}{r|ccccccccc}
$n$ & 2 & 3 & 4 & 5 & 6 & 7 & 8 & 9 & 10 \\
\hline
$\operatorname{lcm}$ & 2 & 6 & 6 & 30 & 120 & 720 &
  15120 & 1164240 & 15135120
\end{tabular}
\end{center}

\section{Open problems}\label{sec:open}

We collect several questions left open by the present
work.

\begin{enumerate}
\item \emph{Description of $h$ and the canonical statistics on different Catalan
    families.} Given the great number of Catalan families~\cite{Stanley15}, it
  would be interesting to see which ones admit a simple description of $h$ and
  the canonical statistics.

\item \emph{Further equidistribution results.}  The involution $h$
  translates more statistics in a meaningful way on some Catalan families. For
  example, let the \emph{descent word} of a permutation $\pi \in S_n$ be
  $w(\pi) \in \{A, D\}^{n-1}$, with $i$th letter $D$ at a descent and $A$ at an
  ascent, and let $\rc$ be the reverse--complement (reverse the word and swap
  $A \leftrightarrow D$). Using induction, one can show that $h$
  reverse-complements the descent word on $\Av(231)$, that is,
  $w(h(\pi)) = \rc(w(\pi))$. Consequently,
  $|w(h(\pi))|_x = |w(\pi)|_{\rc(x)}$ for every factor~$x$. Note that
  $\des(\pi) = |w(\pi)|_D$, $\asc(\pi) = |w(\pi)|_A$, and, similarly, $\peak$,
  $\valley$, $\ddes$, $\dasc$ count the factors $AD$, $DA$, $DD$, $AA$ of the
  descent word. Since
  $\rc$ swaps $D \leftrightarrow A$ and $DD \leftrightarrow AA$ and fixes $AD$
  and $DA$, it follows that $h$ exchanges descents with ascents
  and double descents with double ascents while preserving peaks and valleys.
  Equivalently,
  \[
    (\des, \ddes, \comp, \rmin, \peak, \valley)
    \;\text{ and }\;
    (\asc, \dasc, \rmax, \nonrmin, \peak, \valley)
  \]
  are equidistributed on $\Av(231) \cap S_n$. This is family-specific:
  on $\Av(321)$, $h$ preserves valleys but $\des(h(\pi)) = \asc(\pi)$
  fails already at $n = 3$. It would be interesting to find further
  equidistributions on the other families.

\item \emph{Cycle structure of Donaghey's map.}  By
  Theorem~\ref{thm:M-no-fixedpoints}, $M$ has no fixed points for $n \geq 2$,
  and by Section~\ref{sec:free-product}, $\operatorname{ord}(M|_{C_n})$ is unbounded.
  Knuth asks~\cite[Ex.~17, \S7.2.1.6]{Knuth06} for a characterization of the
  fixed points of $(T \circ R)^2$, which reduces to characterizing the
  $2$-cycles of~$M$.  Beyond the partial results
  in~\cite{PushkarevByzov14,PushkarevByzov19,Shapiro79}, the cycle structure
  remains poorly understood.

\item \emph{Order of $M$ on $C_n$.}  The period of the iterated
  secondary structures equals
  $\operatorname{lcm}_{0 \leq k \leq n}
  \operatorname{ord}(M|_{C_k})$
  (\oeis\ A060114; see the table in
  Section~\ref{sec:iterated}).
  No closed form for
  $\operatorname{ord}(M|_{C_n})$ is known.

\item \emph{A transparent model for both $h$ and
  $\rev$.}
  On triangulations, $h$ is vertex complement
  (Proposition~\ref{prop:vertex-complement}), a global
  geometric operation, but $\rev$ is recursive. On plane
  trees, $\rev$ reverses child lists at every level (a
  global operation), but $h$ is recursive. Since
  $M = h \circ \rev$, a Catalan model on which both $h$
  and $\rev$ have non-recursive descriptions
  would give a direct geometric handle on the cycle
  structure of~$M$. Is there such a model among the
  over two hundred catalogued in~\cite{Stanley15}?

\item \emph{Direct description of $\oplus_3$.}
  The tertiary $\lambda_3$ has a simple description
  on each family (on plane trees, for instance, it adds a
  new leaf as the sole child of the rightmost leaf, so the
  $\oplus_3$-indecomposable trees are those
  whose rightmost leaf is an only child). However,
  $\oplus_3$ involves the full secondary decomposition.
  Is there a direct, non-recursive description
  of~$\oplus_3$ on any concrete Catalan family?

\item \emph{Secondary composition on $\Av(321)$.}
  On $\Av(231)$, the operation $\ominus$ has a direct
  permutation-level description in terms of the last
  entry of~$\pi_1$ (see~Section~\ref{sec:av231}). On
  $\Av(321)$, we define $\ominus$ only through the
  generic recursion on indecomposables
  (see Section~\ref{sec:av321}). Is there a direct description
  there too?
\end{enumerate}

\section{Acknowledgments}
We used Claude (Anthropic) and ChatGPT (OpenAI) to assist with most
aspects of preparing this manuscript. This included exploring through
programming, formulating propositions, developing proofs, and
editing. The authors, however, take full responsibility for the content
of the manuscript.

\appendix

\section{Bijection verification}\label{sec:bij-verify}

The verifications in this appendix and
Appendix~\ref{sec:dict-verify} are included for
completeness.

For each of the seven canonical bijections in
Section~\ref{sec:bijections}, we verify the three identities
(base, indecomposable, decomposable) that characterize the
canonical bijection. Each verification amounts to showing
that the concrete bijection $f$ satisfies
$f(e_1) = e_2$,
$f(\lambda_1(t)) = \lambda_2(f(t))$, and
$f(u \oplus_1 v) = f(u) \oplus_2 f(v)$.
Notation follows Sections \ref{sec:trees}--\ref{sec:ncp}.

\subsection{Triangulations $\to$ plane trees}\label{sec:bij-tri-trees}

Let $f$ send a triangulation of the $(n{+}2)$-gon to a
plane tree as described in Section~\ref{sec:bijections}. The vertex $1$
becomes the root, the vertices adjacent to $1$ (other than
$n{+}2$) become the children of the root, and the
construction recurses into each subpolygon between consecutive
neighbours of $1$.

\textbf{Base.}
The empty triangulation $e_{\text{Tri}}$ is the single
edge $(1,2)$, an edge with no diagonals. On the tree
side, $e_{\text{Trees}} = [\,]$, the single-node tree.
Both are the size-$0$ object, and
$f(e_{\text{Tri}}) = e_{\text{Trees}}$ is immediate (no
vertices adjacent to $1$ besides $2 = n{+}2$). The
triangle $(1,2,3)$ is $\lambda(e)$: vertex $1$ is
adjacent to $2$ and $3$, and the sole subpolygon between them
contains the edge $(2,3)$, which is $e_{\text{Tri}}$. The
image is $[[\,]] = [e] = \lambda(e)$, as required.

\textbf{Indecomposable step.}
We must show
$f(\lambda_{\text{Tri}}(T))
= \lambda_{\text{Trees}}(f(T))$.
The operation $\lambda_{\text{Tri}}(T)$ relabels
$i \to i{+}1$ in a triangulation $T$ of the
$(n{+}2)$-gon and adds vertex $1$ adjacent to $2$ and
$n{+}3$, creating a boundary triangle $(1, 2, n{+}3)$.
In the resulting $(n{+}3)$-gon, vertex $1$ has exactly
one neighbour besides $n{+}3$, namely vertex $2$. The
single subpolygon between them contains the relabeled copy
of $T$. On the tree side, this means the root has exactly
one child, whose subtree is $f(T)$. Thus
$f(\lambda(T)) = [f(T)] = \lambda(f(T))$.

\textbf{Decomposable step.}
We must show
$f(T_1 \oplus_{\text{Tri}} T_2)
= f(T_1) \oplus_{\text{Trees}} f(T_2)$.
The operation $T_1 \oplus T_2$ identifies the edge
$(1, n_1{+}2)$ of $T_1$ with edge $(1,2)$ of $T_2$
(after shifting $T_2$'s labels). In the combined
$(n_1{+}n_2{+}2)$-gon, the neighbours of vertex $1$ are
those from $T_1$ (excluding $n_1{+}2$) followed by those
from the shifted $T_2$. Each subpolygon in
$T_1$ stays intact, and each subpolygon in $T_2$ stays
intact. The subpolygons between consecutive neighbours of $1$
correspond to subtrees hanging from the root, so the
children of the root in $f(T_1 \oplus T_2)$ are exactly
the children of $f(T_1)$ followed by those of $f(T_2)$.
This is $f(T_1) \oplus f(T_2)$, which merges child lists.

\subsection{Plane trees $\to$ Dyck paths}\label{sec:bij-trees-dyck}

Let $f$ send a plane tree to a Dyck path by depth-first
search: each step down an edge produces a $u$-step, each
step back up produces a $d$-step.

\textbf{Base.}
$f([\,]) = \epsilon$: the single-node tree has no edges,
so the traversal produces the empty word.

\textbf{Indecomposable step.}
We must show
$f(\lambda(T)) = \lambda(f(T)) = u f(T) d$.
The tree $\lambda(T) = [T]$ has one child of the root,
with subtree $T$ hanging below it. The traversal first
goes down the root edge (producing $u$), then traverses
the subtree $T$ (producing $f(T)$), then comes back up
the root edge (producing $d$). The result is
$u f(T) d = \lambda(f(T))$.

\textbf{Decomposable step.}
We must show
$f(T_1 \oplus T_2)
= f(T_1) f(T_2) = f(T_1) \oplus f(T_2)$.
The tree $T_1 \oplus T_2$ merges the child lists of the
two roots. The DFS visits the children from left to
right, so it first traverses all children of $T_1$
(producing $f(T_1)$), then all children of $T_2$
(producing $f(T_2)$). The result is the concatenation
$f(T_1) f(T_2)$.

\subsection{Dyck paths $\to$ $2 \times n$ SYT}\label{sec:bij-dyck-syt}

Let $f$ send a Dyck path of semilength $n$ to a
$2 \times n$ SYT by the following rule. Step $i$ is $u$ if and
only if $i$ lies in the top row of the tableau, and
step $i$ is $d$ if and only if $i$ lies in the bottom
row.

\textbf{Base.}
$f(\epsilon) = ([\,], [\,])$: the empty Dyck path
corresponds to the empty tableau.

\textbf{Indecomposable step.}
We must show
$f(\lambda(w)) = \lambda_{\text{SYT}}(f(w))$.
An indecomposable Dyck path $\lambda(w) = uwd$ of
semilength $n$ has step $1$ as $u$ and step $2n$ as $d$.
So $1$ goes in the top row and $2n$ in the bottom row.
The inner steps $2, \ldots, 2n{-}1$ encode the Dyck path
$w$ (shifted by $1$), which by induction gives the SYT
$f(w)$ with entries incremented by $1$. This is exactly
$\lambda_{\text{SYT}}(f(w))$: since $f(w)$ is a
$2 \times (n{-}1)$ SYT, $\lambda_{\text{SYT}}$ prepends
$1$ to the top row and appends
$2(n{-}1){+}2 = 2n$ to the bottom row.

\textbf{Decomposable step.}
We must show
$f(w_1 w_2) = f(w_1) \oplus f(w_2)$.
The concatenation $w_1 w_2$ has the first $2n_1$ steps
determined by $w_1$ and the last $2n_2$ steps by $w_2$.
The first $2n_1$ steps produce $f(w_1)$, and the last
$2n_2$ steps produce $f(w_2)$ with entries shifted by
$2n_1$. Placing these side by side gives
$f(w_1) \oplus f(w_2)$.

\subsection{Plane trees $\to$ binary trees}\label{sec:bij-trees-binary}

Let $f$ be the left-child/right-sibling encoding: the
leftmost edge of the plane tree becomes the root of the
binary tree, its subtree becomes the left child, and the
remaining siblings form the right-sibling chain.

\textbf{Base.}
$f([\,]) = \varnothing$: the single-node tree (no edges)
maps to the empty binary tree.

\textbf{Indecomposable step.}
We must show
$f(\lambda(T)) = \lambda_{\text{Bin}}(f(T))
= (f(T), \varnothing)$.
The tree $\lambda(T) = [T]$ has one child. There is a
single edge $\eta_0$ from root to child, with subtree
$T_1 = T$ below it and no siblings ($T_2$ is a single
node). The binary tree has root corresponding to edge
$\eta_0$, left subtree $f(T)$, and right subtree
$\varnothing$. This is
$(f(T), \varnothing) = \lambda(f(T))$.

\textbf{Decomposable step.}
We must show
$f(T_1 \oplus T_2)
= f(T_1) \oplus_{\text{Bin}} f(T_2)$.
The tree $T_1 \oplus T_2$ merges child lists. The
leftmost edge of $T_1$ becomes the root of the binary
tree, with its subtree as left child. The remaining
children of $T_1$ followed by all children of $T_2$ form
the right-sibling chain. In the binary tree, this means
$f(T_2)$ is attached at the end of the right spine of
$f(T_1)$, which is exactly
$f(T_1) \oplus_{\text{Bin}} f(T_2)$.

\subsection{Plane trees $\to$ non-crossing
  partitions}\label{sec:bij-trees-ncp}

Let $f$ send a plane tree with $n$ edges to a
non-crossing partition of $[n]$ as follows: label the
edges of $T$ in postorder from $1$ to $n$. The labels
along the leftmost path form the first block, and the
construction recurses on the remaining subtrees.

\textbf{Base.}
$f([\,]) = [\,]$: the single-node tree has no edges,
giving the empty partition.

\textbf{Indecomposable step.}
We must show
$f(\lambda(T)) = \lambda_{\text{NCP}}(f(T))$.
The tree $\lambda(T) = [T]$ adds a new root edge. In
postorder, the new root edge is the last visited, so it
receives label $n{+}1$. This edge lies on the leftmost
path, so $n{+}1$ is adjoined to the block containing
$1$. On the NCP side, $\lambda_{\text{NCP}}(P)$ adjoins
$n{+}1$ to the block containing $1$. Thus
$f(\lambda(T)) = \lambda(f(T))$.

\textbf{Decomposable step.}
We must show
$f(T_1 \oplus T_2) = f(T_1) \oplus f(T_2)$.
The tree $T_1 \oplus T_2$ merges child lists. In
postorder, $T_1$'s subtree is processed first, receiving
labels $1, \ldots, n_1$, then $T_2$'s subtree receives
labels $n_1{+}1, \ldots, n_1{+}n_2$. The resulting
partition is $f(T_1) \cup (f(T_2) + n_1)$, which is
$f(T_1) \oplus f(T_2)$: the shifted union.

\subsection{Plane trees $\to$ $\Av(231)$}\label{sec:bij-trees-231}

Let $f$ send a plane tree to a $231$-avoiding
permutation: label the edges in postorder (as in
Section~\ref{sec:bij-trees-ncp}), then read the labels in
preorder.

\textbf{Base.}
$f([\,]) = \epsilon$: no edges, empty permutation.

\textbf{Indecomposable step.}
We must show
$f(\lambda(T)) = \lambda(\,f(T)) = (n{+}1) f(T)$.
The tree $\lambda(T) = [T]$ has a new root edge. In
postorder it is labeled $n{+}1$ (the largest label).
The preorder reading visits the root edge first, so the
permutation begins with $n{+}1$, followed by the
preorder reading of $T$'s edges, which by induction is
$f(T)$. Thus
$f([T]) = (n{+}1) f(T) = \lambda(f(T))$.

\textbf{Decomposable step.}
We must show
$f(T_1 \oplus T_2) = f(T_1) \oplus f(T_2)$.
The postorder labeling assigns $1, \ldots, n_1$ to
$T_1$'s edges and $n_1{+}1, \ldots, n_1{+}n_2$ to
$T_2$'s edges. The preorder reading visits $T_1$'s
subtree first (producing $f(T_1)$, a permutation of
$\{1, \ldots, n_1\}$), then $T_2$'s subtree (producing
$f(T_2)$ shifted by $n_1$, a permutation of
$\{n_1{+}1, \ldots, n_1{+}n_2\}$). The result is the
direct sum $f(T_1) \oplus f(T_2)$.

\subsection{Plane trees $\to$ $\Av(321)$}\label{sec:bij-trees-321}

Let $f$ send a plane tree to a $321$-avoiding
permutation: label the edges in postorder, then apply
the slot-filling traversal described in
Section~\ref{sec:bijections}.

\begin{lemma}\label{lem:retlmax}
In the slot-filling traversal of any tree $S$, the
labels written on return (up) steps are placed in
strictly increasing order, and each occupies a
left-to-right maximum position of $f(S)$.
\end{lemma}

\begin{proof}
Consider a non-leaf edge $\eta$ with label $\ell$. The
edges whose labels are placed before the return of
$\eta$ fall into two groups: (i) edges in $\eta$'s
subtree, and (ii) edges in left-sibling subtrees at
every ancestor level of $\eta$. Every such edge $\eta'$
is strictly smaller than $\eta$ in postorder. For
group (i), $\eta'$ is a descendant of $\eta$ and so
precedes $\eta$ in postorder. For group (ii), $\eta'$
belongs to a left-sibling subtree of some ancestor of
$\eta$, which is fully processed before $\eta$'s
ancestor chain. Hence $\ell' < \ell$.

When $\eta$ returns, its label $\ell$ is placed in the
leftmost available slot. All positions to the left of
this slot have already been filled with labels from
groups (i) and (ii), each less than $\ell$, so $\ell$
occupies a left-to-right maximum position. Since the set
of labels placed before each successive return only
grows, the return labels are placed in strictly
increasing order.
\end{proof}

\begin{lemma}\label{lem:retfill}
A position $j$ in $\pi = f(T)$ is a return-fill position
if and only if it is a left-to-right maximum but not a
right-to-left minimum.
\end{lemma}

\begin{proof}
That a return-fill position is a left-to-right maximum
is Lemma~\ref{lem:retlmax}. For the ``not a
right-to-left minimum'' part: a non-leaf edge $\eta$
with label $\ell$ returning to position $j$ has at
least one descendant $\eta'$ with $\ell' < \ell$
filling some $j' > j$, so $j$ is not a right-to-left
minimum.

Conversely, suppose a leaf edge $\eta$ with label
$\ell$ fills a position $j$ that is a left-to-right
maximum. Every edge $\eta'$ filling a position
$j' > j$ is either in a right-sibling subtree of an
ancestor of $\eta$, or is an ancestor of $\eta$
returning to a slot $> j$. In both cases $\eta$
precedes $\eta'$ in postorder, so $\ell < \ell'$,
making $j$ a right-to-left minimum.
\end{proof}

\textbf{Verification that $f$ is canonical.}

\textbf{Base.}
$f([\,]) = \epsilon$: no edges, empty permutation.

\textbf{Indecomposable step.}
We must show
$f(\lambda(T)) = \lambda(\,f(T))$.

\emph{Case $T = e$.} The tree $\lambda(e) = [e]$ has one
leaf edge labeled $1$. Slot-filling places it
immediately, giving $f([e]) = 1$. On the other side,
$\lambda(f(e)) = \lambda(\epsilon) = 1$.

\emph{Case $|T| \geq 1$.} Write $\pi = f(T)$,
$n = |T|$, $i = \pi^{-1}(1)$, and
$T' = \lambda(T) = [T]$. Positions in all permutations
are numbered $1, 2, \ldots$ In the postorder labeling on
$T'$, the new root edge $\eta_*$ is visited last and
therefore receives label $n{+}1$, while every edge of
$T$ keeps its original label.

By Lemma~\ref{lem:retfill}, the return-fill positions
of $\pi$ are the left-to-right maxima that are not
right-to-left minima. Since $\pi(i) = 1$ is the global
minimum, $i$ is a right-to-left minimum, so $i$ is not
a return-fill. The DFS traversal descends the non-leaf
edges on the leftmost path (positions $1, 2, \ldots,
i{-}1$) before reaching the leftmost leaf (position $i$,
label $1$), so the return-fills $< i$ are exactly
$\{1, 2, \ldots, i{-}1\}$. Write $c_1 < \cdots < c_q$
for the return-fills $> i$ and set $k = i - 1 + q$.

Run the slot-filling on $T'$. The initial descent along
$\eta_*$ creates one extra slot at position $1$; the
remaining events are those of $T$ shifted right by one.
The $k{+}1$ slots in $\pi' = f(T')$ therefore sit at
$1, 2, \ldots, i, c_1{+}1, \ldots, c_q{+}1$ and are
filled left to right with
$\pi(1), \ldots, \pi(i{-}1),
\pi(c_1), \ldots, \pi(c_q), n{+}1$.
Each non-slot position $l{+}1$ carries $\pi(l)$.

Now compare with $\lambda(\pi)$. Inserting $n{+}1$ at
position $i$ and shifting positions $\geq i$ right by
one gives the same effect on non-return positions:
position $j < i$ retains $\pi(j)$, and position $l{+}1$
for $l \geq i$ retains $\pi(l)$. The boxed positions of
$\lambda$ (see Section~\ref{sec:av321}) are
$B = \{i, c_1{+}1, \ldots, c_q{+}1\}$: position $i$
(newly inserted) together with the positions $> i$
that are left-to-right maxima but not right-to-left
minima, each shifted by $1$. The cyclic
left-shift at $B$ sends $n{+}1$ from position $i$ to
$c_q{+}1$, and $\pi(c_j)$ from $c_j{+}1$ to
$c_{j-1}{+}1$ (or to $i$ when $j = 1$). This matches
the slot-fill values above, so
$\pi' = \lambda(\pi) = \lambda(f(T))$.

\textbf{Decomposable step.}
We must show
$f(T_1 \oplus T_2) = f(T_1) \oplus f(T_2)$.
The tree $T_1 \oplus T_2$ concatenates the child lists
of $T_1$ and $T_2$ under the root. In postorder $T_1$'s
subtree is processed before $T_2$'s, assigning labels
$1, \ldots, n_1$ to $T_1$'s edges and
$n_1{+}1, \ldots, n_1{+}n_2$ to $T_2$'s edges; within
each subtree the relative labeling is identical to that
of $T_i$ alone. The slot-filling procedure processes children left
to right, so it fills $T_1$'s slots first using labels
from $\{1, \ldots, n_1\}$, producing $f(T_1)$ in
positions $1, \ldots, n_1$. It then fills $T_2$'s slots
using labels from $\{n_1{+}1, \ldots, n_1{+}n_2\}$,
producing a shifted copy of $f(T_2)$ in positions
$n_1{+}1, \ldots, n_1{+}n_2$. No cross-talk occurs
because the position ranges and label sets are disjoint.
The result is the direct sum $f(T_1) \oplus f(T_2)$.

\section{Dictionary verification}\label{sec:dict-verify}

We verify that each entry in the dictionary table
of Section~\ref{sec:dictionary} satisfies the canonical recursion from
Section~\ref{sec:equidistribution}. For each family
$\mathcal{C}$ and each native statistic $s$, we check the
three cases (base $e$, indecomposable $\lambda$, and
decomposable $\oplus$), showing that $s$ satisfies the
same recursion as the corresponding canonical
statistic. Once
this is verified, the transport proposition
(Proposition~\ref{prop:transport})
transfers the equidistribution result to $\mathcal{C}$
automatically. The base-case caveat
(Remark~\ref{rem:base-case}) applies to the last two
rows of the dictionary.

\begin{lemma}\label{lem:transport-internal}
If the canonical bijection sends $\leaves$ to a native
statistic $L$, then it sends $\internal$ to $n + 1 - L$
for $n \geq 1$ (with $\internal(e) = 0$ at $n = 0$).
This follows from
$\internal(t) = |t| + 1 - \leaves(t)$ and size
preservation.
\end{lemma}

By this lemma, the $\internal$ row of the dictionary is
determined by the $\leaves$ row, so the per-family
verifications below omit it. All definitions of $e$,
$\lambda$, and $\oplus$ are in Section~\ref{sec:families}.

\subsection{Dyck paths}\label{sec:dict-dyck}

\textbf{$\boldsymbol{\subtrees \to \comp}$.}
The statistic $\comp$ counts returns to the $x$-axis.
For $\epsilon$, $\comp(\epsilon) = 0$. For $uwd$: the
path stays strictly above the $x$-axis between the
initial $u$ and the final $d$, so $\comp(uwd) = 1$.
For $w_1 w_2$: since $w_1$ ends at height $0$ and $w_2$
begins at height $0$, the returns split additively:
$\comp(w_1 w_2) = \comp(w_1) + \comp(w_2)$.

\textbf{$\boldsymbol{\rpath \to \rcomp}$.}
The statistic $\rcomp$ is the length of the maximal run
of $d$-steps at the end. For $\epsilon$,
$\rcomp(\epsilon) = 0$. For $uwd$: the path ends with
$d$; if $w = \epsilon$ the final run is $d$ of length
$1$, and if $w \neq \epsilon$ then $w$ ends with $d$ so
the final run extends by one step. Thus
$\rcomp(uwd) = 1 + \rcomp(w)$.
For $w_1 w_2$: the final run of $d$-steps is determined
entirely by $w_2$ (since $w_2$ starts with $u$, the run
cannot extend into $w_1$), so
$\rcomp(w_1 w_2) = \rcomp(w_2)$.

\textbf{$\boldsymbol{\leaves \to \peak}$.}
See the model proof in Section~\ref{sec:transport}.

\subsection{Triangulations}\label{sec:dict-tri}

\textbf{$\boldsymbol{\subtrees \to \comp}$.}
The statistic $\comp$ counts edges incident to vertex
$1$, minus $1$. For $e = (1,2)$, vertex $1$ has one
neighbour, so $\comp(e) = 0$. For $\lambda(T)$: vertex
$1$ is adjacent to exactly $2$ and $n{+}3$, giving
$\comp = 1$. For $T_1 \oplus T_2$: the neighbours of $1$
in $T_1$ (other than $n_1{+}2$) together with the
neighbours of $1$ in $T_2$ (shifted) give
$\comp(T_1 \oplus T_2)
= \comp(T_1) + \comp(T_2)$.

\textbf{$\boldsymbol{\rpath \to \rcomp}$.}
The statistic $\rcomp$ counts edges incident to vertex
$n{+}2$, minus $1$. For $e$, $\rcomp(e) = 0$. For
$\lambda(T)$: the last vertex of the $(n{+}3)$-gon is
$n{+}3$, which inherits the (shifted) adjacencies of
vertex $n{+}2$ in $T$ and gains vertex $1$ as an
additional neighbour, so
$\rcomp(\lambda(T)) = 1 + \rcomp(T)$.
For $T_1 \oplus T_2$: the last vertex $n_1{+}n_2{+}2$
comes from $T_2$ (shifted), and its adjacencies are
entirely determined by $T_2$, so
$\rcomp(T_1 \oplus T_2) = \rcomp(T_2)$.

\textbf{$\boldsymbol{\leaves \to \tr}$.}
The statistic $\tr$ counts indices
$i \in \{2, \ldots, n{+}1\}$ such that $i$ and $i{+}1$
share a common neighbour $j < i$. For $e$,
$\tr(e) = 0$ (Remark~\ref{rem:base-case}).

For $\lambda(T)$: the shift $i \to i{+}1$ sends each
qualifying index $i$ in $T$ to $i{+}1$ in $\lambda(T)$,
preserving the triangle structure. The only new candidate
is $i = 2$: vertices $2$ and $3$ share vertex $1$ if and
only if $1$ is adjacent to $3$, which holds if and only
if $n{+}3 = 3$, that is,\
$T = e$. When $T = e$, $\lambda(e)$ is the triangle
$(1,2,3)$ and $\tr = 1 = \leaves(e)$. When $T \neq e$,
vertex $1$ is adjacent only to $2$ and $n{+}3 \geq 4$,
so $i = 2$ does not qualify, and
$\tr(\lambda(T)) = \tr(T)$.

For $T_1 \oplus T_2$: the gluing identifies an edge
of $T_1$ with an edge of $T_2$ without creating any
new triangle. Hence the triangles of the combined polygon
are the disjoint union of those of $T_1$ and the
shifted triangles of $T_2$, and the qualifying indices
split between the two ranges:
$\tr(T_1 \oplus T_2) = \tr(T_1) + \tr(T_2)$.

\subsection{$231$-avoiding permutations}\label{sec:dict-231}

\textbf{$\boldsymbol{\subtrees \to \comp}$.}
The statistic $\comp$ counts components,
where a $231$-avoiding permutation is indecomposable when it
starts with its maximum. For $e$,
$\comp(\epsilon) = 0$. For
$\lambda(\pi) = (n{+}1) \pi$: the permutation
starts with its maximum $n{+}1$, so it is indecomposable
and $\comp(\lambda(\pi)) = 1$. For
$\pi_1 \oplus \pi_2$: the prefix $\pi_1$ permutes
$\{1, \ldots, k\}$ and $\pi_2$ (shifted) permutes
$\{k{+}1, \ldots, k{+}m\}$, so the components of the
direct sum are exactly those of $\pi_1$ followed by
those of $\pi_2$:
$\comp(\pi_1 \oplus \pi_2)
= \comp(\pi_1) + \comp(\pi_2)$.

\textbf{$\boldsymbol{\rpath \to \rmax}$.}
The statistic $\rmax$ counts right-to-left maxima. For
$e$, $\rmax(\epsilon) = 0$. For
$\lambda(\pi) = (n{+}1) \pi$: since $n{+}1$ is
prepended and all values in $\pi$ are smaller, the
right-to-left maxima of $\lambda(\pi)$ are exactly those
of $\pi$ together with $n{+}1$ itself. Thus
$\rmax(\lambda(\pi)) = 1 + \rmax(\pi)$.
For $\pi_1 \oplus \pi_2$: every value of $\pi_2$
(shifted) exceeds every value of $\pi_1$, so a
right-to-left maximum in $\pi_1$ is blocked by the
values of $\pi_2$ to its right (unless $\pi_2$ is empty,
but then we are not in the $\oplus$ case). Hence the
right-to-left maxima of the direct sum are exactly those
of the shifted $\pi_2$:
$\rmax(\pi_1 \oplus \pi_2) = \rmax(\pi_2)$.

\textbf{$\boldsymbol{\leaves \to \rmin}$.}
The statistic $\rmin$ counts right-to-left minima. For
$e$, $\rmin(\epsilon) = 0$ (Remark~\ref{rem:base-case}).

For $\lambda(\pi) = (n{+}1) \pi$: the prepended
$n{+}1$ is the global maximum and is never a
right-to-left minimum. The remaining entries are $\pi$
in the same relative order, so their right-to-left
minima are unchanged:
$\rmin(\lambda(\pi)) = \rmin(\pi)$.

For $\pi_1 \oplus \pi_2$: the values of $\pi_1$ lie in
$\{1, \ldots, k\}$ and those of $\pi_2$ (shifted) lie
in $\{k{+}1, \ldots, k{+}m\}$. Every value in $\pi_1$
is smaller than every value to its right coming from
$\pi_2$, so right-to-left minima within $\pi_1$ remain
right-to-left minima of the whole permutation.
Right-to-left minima within the shifted $\pi_2$ are
determined by $\pi_2$ alone (no smaller values follow).
Thus $\rmin(\pi_1 \oplus \pi_2)
= \rmin(\pi_1) + \rmin(\pi_2)$.

\subsection{$321$-avoiding permutations}\label{sec:dict-321}

For $\Av(321)$, $e = \epsilon$,
$\lambda(\pi)$ inserts $n{+}1$ at the position of $1$
and cyclically shifts the boxed positions (see Section~\ref{sec:av321} for
the full description), and $\pi_1 \oplus \pi_2$ is the
direct sum. The closure and bijectivity of $\lambda$ are
proved in Section~\ref{sec:av321}. The arguments below use only the
following structural property: the right-to-left
minima of $\pi$ are never boxed.

\textbf{$\boldsymbol{\subtrees \to \comp}$.}
The statistic $\comp$ counts the components of the
direct-sum decomposition. For $e$,
$\comp(\epsilon) = 0$. For $\lambda(\pi)$: by
Proposition~\ref{prop:321-bijectivity}, $\sigma = \lambda(\pi)$
is indecomposable, so $\comp(\lambda(\pi)) = 1$. For the direct sum:
$\comp(\pi_1 \oplus \pi_2)
= \comp(\pi_1) + \comp(\pi_2)$, by the same argument
as for $\Av(231)$.

\textbf{$\boldsymbol{\rpath \to \rcomp}$.}
The statistic $\rcomp(\pi) = n - \pi(n) + 1$. For $e$,
$\rcomp(\epsilon) = 0$.

For $\lambda(\pi)$: the operation $\lambda$ inserts
$n{+}1$ at the position of $1$ and cyclically shifts the
boxed positions. The last entry of $\lambda(\pi)$
satisfies $\lambda(\pi)(n{+}1) = \pi(n)$: the value
$\pi(n)$ is always a right-to-left minimum of $\pi$
and hence never boxed, so the cyclic shift does not
affect the last position. Therefore
$\rcomp(\lambda(\pi))
= (n{+}1) - \pi(n) + 1 = (n - \pi(n) + 1) + 1
= 1 + \rcomp(\pi)$.

For $\pi_1 \oplus \pi_2$ with $|\pi_1| = k$: the last
entry is $\pi_2(m) + k$ where $m = |\pi_2|$. Then
$\rcomp(\pi_1 \oplus \pi_2)
= (k{+}m) - (\pi_2(m) + k) + 1
= m - \pi_2(m) + 1 = \rcomp(\pi_2)$.

\textbf{$\boldsymbol{\leaves \to \rmin}$.}
The statistic $\rmin$ counts right-to-left minima. For
$e$, $\rmin(\epsilon) = 0$ (Remark~\ref{rem:base-case}).

For $\lambda(\pi)$: the inserted value $n{+}1$ is the
global maximum and cannot be a right-to-left minimum.
The cyclic shift permutes the values at the boxed
positions $b_0, \ldots, b_m$. Both before and after
the shift, these values are left-to-right maxima of
$\sigma$, each exceeding the suffix minimum at its
position (since each boxed position is not a
right-to-left minimum, there is a smaller value in
its suffix at a non-boxed position). Hence the shift does not create or destroy
any right-to-left minimum, giving
$\rmin(\lambda(\pi)) = \rmin(\pi)$.

For $\pi_1 \oplus \pi_2$ with $|\pi_1| = k$: values of
$\pi_1$ lie in $\{1, \ldots, k\}$ and values of the
shifted $\pi_2$ lie in $\{k{+}1, \ldots, k{+}m\}$. A
position $j \leq k$ is a right-to-left minimum of the
whole permutation if and only if it is one in $\pi_1$:
the $\pi_2$ values to its right are all
$> k \geq \pi_1(j)$, so they impose no new constraint,
and witnesses within $\pi_1$ are preserved. A position
$j > k$ is a right-to-left minimum of the whole if and
only if it
is one in $\pi_2$: all positions to its right also lie
in the $\pi_2$ block. Hence
$\rmin(\pi_1 \oplus \pi_2)
= \rmin(\pi_1) + \rmin(\pi_2)$.

\subsection{Binary trees}\label{sec:dict-binary}

\textbf{$\boldsymbol{\subtrees \to \comp}$.}
The statistic $\comp$ is the number of nodes on the
right spine from the root (that is,\ the number of
indecomposable summands). For $\varnothing$,
$\comp(\varnothing) = 0$. For
$\lambda(T) = (T, \varnothing)$: the right spine has
one node, so $\comp = 1$. For $B_1 \oplus B_2$: the
right spine of the combined tree consists of the right
spine of $B_1$ extended by that of $B_2$, so
$\comp(B_1 \oplus B_2)
= \comp(B_1) + \comp(B_2)$.

\textbf{$\boldsymbol{\rpath \to \rcomp}$.}
The statistic $\rcomp(B)$ is computed by iteratively
peeling right spines: go to the last node on the right
spine, replace $B$ by that node's left subtree, and
repeat; $\rcomp$ counts the number of iterations before
reaching $\varnothing$. For $\varnothing$,
$\rcomp(\varnothing) = 0$. For
$\lambda(T) = (T, \varnothing)$: the right spine
consists of the root alone, whose left subtree is $T$,
so one iteration reduces to $T$:
$\rcomp((T, \varnothing)) = 1 + \rcomp(T)$.
For $B_1 \oplus B_2$: the last node on the combined
right spine is the last node of $B_2$'s right spine,
with the same left subtree, so
$\rcomp(B_1 \oplus B_2) = \rcomp(B_2)$.

\textbf{$\boldsymbol{\leaves \to \redge}$.}
We define
$\redge(B) = 1 + (\text{number of right edges of } B)$,
where a right edge is an edge from a node to a nonempty
right child. Equivalently, $\redge$ satisfies the
three-case recursion
\begin{align*}
  \redge(\varnothing) &= 1 \\
  \redge((L, \varnothing)) &= \redge(L) \\
  \redge((L, R)) &= \redge(L) + \redge(R)
    \quad\text{for } R \neq \varnothing
\end{align*}
matching the canonical recursion for $\leaves$.
Indeed: for $\lambda(T) = (T, \varnothing)$, the root has
no right child, so $\lambda$ introduces no new right
edge: $\redge(\lambda(T)) = \redge(T)$. For
$B_1 \oplus B_2$: $\oplus$ creates exactly one new right
edge (from the last node of $B_1$'s right spine to the
root of $B_2$), preserving the internal edges. Writing
$r_i$ for the number of right edges of $B_i$:
$\redge(B_1 \oplus B_2)
= 1 + r_1 + 1 + r_2
= \redge(B_1) + \redge(B_2)$.

\subsection{$2 \times n$ SYT}\label{sec:dict-syt}

\textbf{$\boldsymbol{\subtrees \to \comp}$.}
The statistic $\comp$ counts components:
positions $k$ where the first $k$ columns contain exactly
$\{1, \ldots, 2k\}$. For $e$, $\comp(e) = 0$. For
$\lambda(T)$: since $1$ is in position $(1,1)$ and
$2n{+}2$ is in position $(2,n{+}1)$, no proper prefix
of columns contains $\{1, \ldots, 2j\}$ (the entry $1$
prevents early closure and $2n{+}2$ prevents premature
completion), so $\comp(\lambda(T)) = 1$.
For $T_1 \oplus T_2$: the blocks are determined
independently in the two halves, so
$\comp(T_1 \oplus T_2)
= \comp(T_1) + \comp(T_2)$.

\textbf{$\boldsymbol{\rpath \to \rcomp}$.}
The statistic $\rcomp$ is the length of the maximal
consecutive segment at the right end of the bottom row:
the largest $k$ such that the last $k$ entries of the
bottom row, read right to left, are
$2n, 2n{-}1, \ldots, 2n{-}k{+}1$. For $e$,
$\rcomp(e) = 0$. For $\lambda(T)$: the bottom row ends
with $2n{+}2$ (newly appended). The previous bottom
entry is the last entry of the shifted inner bottom row,
which is $2n' + 1$ (where $n' = |T|$) if it ended the
inner consecutive segment. Since $\lambda$ places
$2n{+}2$ at the end and shifts the inner entries by $1$,
the consecutive segment extends by one:
$\rcomp(\lambda(T)) = 1 + \rcomp(T)$.
For $T_1 \oplus T_2$: the right end of the bottom row is
entirely determined by $T_2$ (shifted), so
$\rcomp(T_1 \oplus T_2) = \rcomp(T_2)$.

\textbf{$\boldsymbol{\leaves \to \noncons}$.}
The statistic $\noncons$ counts entries $i$ in the top
row such that $i{+}1$ is not in the top row. For $e$,
$\noncons(e) = 0$ (Remark~\ref{rem:base-case}).

For $\lambda(T)$: the top row of $\lambda(T)$ is
$\{1\} \cup \{t_j + 1 : t_j \in \text{top}(T)\}$.
When $T = e$, $\lambda(e) = ([1], [2])$, so entry $1$
has $2$ in the bottom row and
$\noncons = 1 = \leaves(e)$. When $T \neq e$, the top
row of $T$ contains $1$, so the top row of $\lambda(T)$
contains both $1$ and $2 = 1 + 1$; thus entry $1$ does
not contribute to $\noncons$. The shifted entries
preserve the noncons count:
$\noncons(\lambda(T)) = \noncons(T)$.

For $T_1 \oplus T_2$: entries from $T_1$ and $T_2$
occupy disjoint value ranges ($\{1, \ldots, 2n_1\}$ and
$\{2n_1{+}1, \ldots, 2n_1{+}2n_2\}$). The largest top
entry of $T_1$ is at most $2n_1 - 1$ (since $2n_1$ is
always in the bottom row), so for every top entry $i$ of
$T_1$, the successor $i + 1 \leq 2n_1$ lies within
$T_1$'s range. Likewise, for every top entry $i$ of the
shifted $T_2$, $i + 1$ lies within $T_2$'s range. Thus
the noncons contributions from $T_1$ and $T_2$ are
independent:
$\noncons(T_1 \oplus T_2)
= \noncons(T_1) + \noncons(T_2)$.

\subsection{Non-crossing partitions}\label{sec:dict-ncp}

\textbf{$\boldsymbol{\subtrees \to \comp}$.}
The statistic $\comp$ counts components. For
$e$, $\comp(e) = 0$. For $\lambda(P)$: since $1$ and
$n{+}1$ belong to the same block, $\lambda(P)$ is
indecomposable, so $\comp(\lambda(P)) = 1$.
For $P_1 \oplus P_2$: the blocks of $P_1$ occupy
$\{1, \ldots, n_1\}$ and those of $P_2$ (shifted)
occupy $\{n_1{+}1, \ldots, n_1{+}n_2\}$, so the
components of the two partitions
contribute independently:
$\comp(P_1 \oplus P_2)
= \comp(P_1) + \comp(P_2)$.

\textbf{$\boldsymbol{\rpath \to \rcomp}$.}
The statistic $\rcomp$ counts right-to-left maxima in
the flattened form of $P$, where the flattened form
lists the elements of each block in decreasing order,
with blocks sorted by their minimal element. For $e$,
$\rcomp(e) = 0$.

For $\lambda(P)$: the operation adjoins $n{+}1$ to the
block of $1$, making $n{+}1$ the first element of that
block in the flattened form (since it is the new maximum
of the block). The remaining elements appear in the same
relative order. Since $n{+}1$ exceeds all other elements,
it is always a right-to-left maximum in the flattened
sequence, adding exactly one to the count. Thus
$\rcomp(\lambda(P)) = 1 + \rcomp(P)$.

For $P_1 \oplus P_2$: every element of the shifted $P_2$
exceeds every element of $P_1$, so no element of $P_1$
can be a right-to-left maximum in the combined flattened
form. The right-to-left maxima are exactly those within
$P_2$'s portion:
$\rcomp(P_1 \oplus P_2) = \rcomp(P_2)$.

\textbf{$\boldsymbol{\leaves \to \blocks}$.}
The statistic $\blocks$ counts blocks. For $e$,
$\blocks(e) = 0$ (Remark~\ref{rem:base-case}).

For $\lambda(P)$: when $P = e$,
$\lambda(e) = \{\{1\}\}$ has one block and
$\leaves(e) = 1$. When $P \neq e$, $\lambda$ merges
$n{+}1$ into the existing block of $1$ without creating
or destroying any block, so
$\blocks(\lambda(P)) = \blocks(P)$.

For $P_1 \oplus P_2$: the blocks are the disjoint union
of $P_1$'s blocks and $P_2$'s (shifted) blocks, so
$\blocks(P_1 \oplus P_2)
= \blocks(P_1) + \blocks(P_2)$.

\end{document}